\documentclass[11pt]{article}
\usepackage[nottoc]{tocbibind}
\usepackage[T1]{fontenc}
\usepackage[utf8]{inputenc}
\usepackage{enumerate}
\usepackage{changes}
\usepackage{hyperref}
\usepackage{amsmath}
\usepackage{amsfonts}
\usepackage{amssymb}
\usepackage{amsthm}
\usepackage{newlfont}
\usepackage{mathtools}
\usepackage{esint}
\usepackage[top=1.2in, bottom=1.3in, left=1in, right=1in]{geometry}

\usepackage{authblk}
\usepackage{xcolor}
\usepackage{thmtools}
\usepackage{tikz}

\newcommand{\vertiii}[1]{{\left\vert\kern-0.25ex\left\vert\kern-0.25ex\left\vert #1 
    \right\vert\kern-0.25ex\right\vert\kern-0.25ex\right\vert}}

\theoremstyle{plain}
\newtheorem{theorem}{Theorem}[section]
\newtheorem{proposition}[theorem]{Proposition}
\newtheorem{corollary}[theorem]{Corollary}
\newtheorem{lemma}[theorem]{Lemma}

\theoremstyle{definition}
\newtheorem{definition}[theorem]{Definition}

\newtheorem{remark}[theorem]{Remark}

\newcommand{\RR}{\mathbb{R}}
\newcommand{\NN}{\mathbb{N}}
\newcommand{\ZZ}{\mathbb{Z}}

\newcommand{\cC}{\mathcal{C}}
\newcommand{\drm}{\mathrm{d}}
\newcommand{\euler}{\mathrm{e}}

\newcommand{\cE}{\mathcal{E}}

\DeclareMathOperator{\supp}{supp}

\DeclareMathOperator{\Deg}{Deg}

\DeclareMathOperator{\spec}{spec}

\DeclareMathOperator{\arsinh}{arsinh}

\DeclareMathOperator{\Eins}{\mathbf{1}}
\DeclareMathOperator{\diam}{diam}
\newcommand{\tvert}[1]{{\left\vert\kern-0.25ex\left\vert\kern-0.25ex\left\vert #1 
    \right\vert\kern-0.25ex\right\vert\kern-0.25ex\right\vert}}

\newcommand{\eat}[1]{}

\let\oldint\int
\renewcommand{\int}{\oldint\limits}

\newcommand{\Hmm}[1]{\leavevmode{\marginpar{\tiny%
			$\hbox to 0mm{\hspace*{-0.5mm}$\leftarrow$\hss}%
			\vcenter{\vrule depth 0.1mm height 0.1mm width \the\marginparwidth}%
			\hbox to
			0mm{\hss$\rightarrow$\hspace*{-0.5mm}}$\\\relax\raggedright #1}}}

\begin{document}

\title{Gaussian upper bounds for averaged heat semigroups on graphs
}
\author{Christian Rose\thanks{christian.rose\@ uni-potsdam.de. \\ Support by the DFG (grant no.:540199605) is gratefully acknowledged.}
}
\affil[]{Institut f\"ur Mathematik, Universit\"at Potsdam, 		14476  Potsdam, Germany}
\date{\today}
\maketitle

\begin{abstract}
Characterizations of pointwise Gaussian upper bounds on graphs with possibly unbounded geometry in terms of localized functional inequalities contain errors depending on the vertex degree. We introduce a new space-time averaged form of the heat semigroup in terms of time-averaged $\ell^{q/(q-1)}-\ell^q$-estimates and obtain Gaussian upper bounds from large-scale Faber-Krahn inequalities which avoid such errors. A main analytic ingredient is an integrated version of Davies' method which yields off-diagonal estimates for this averaged quantity, and which tends to pointwise Gaussian bounds as time tends to infinity. Conversely, on large scales volume doubling and Gaussian bounds on this averaged heat semigroup norms imply relative Faber-Krahn inequalities for subsets of prescribed relative measure. Our proof is based on a lower bound for Dirichlet eigenvalues in terms of these averages. The Faber-Krahn dimension is scale-dependent, but converges to the doubling dimension for increasing radii.
This gives a characterization of asymptotic Gaussian heat kernel behavior in terms of Faber-Krahn inequalities. 
\end{abstract}
{\small
\tableofcontents}
\section{Introduction and main results}

Gaussian upper bounds for heat kernels are fundamentally linked to the geometry of a metric measure space through functional inequalities involving the associated Laplace operator. 
A basic difficulty arises for the corresponding characterizations when the underlying space is a discrete graph with unbounded vertex degree: pointwise Gaussian estimates contain error terms involving the degree which are absent on large-scales.
In particular, they obstruct a degree-independent large-scale characterization of Gaussian behavior of the heat kernel.

This suggests to find characterizations of heat kernel upper bounds which are averaged both in space and in time. The resulting quantities should retain the off-diagonal Gaussian structure while allowing the local degree obstruction appearing on graphs to disappear.

Our purpose is to solve this problem on graphs by developing a theory on Gaussian bounds for averaged heat semigroups and to establish its connection with relative Faber-Krahn inequalities.

Conceptually, our main result is a large-scale characterization of Gaussian upper bounds on space-time averaged heat semigroups in terms of relative Faber-Krahn inequalities without degree-dependent errors. More precisely, Faber-Krahn yields such averaged bounds which asymptotically correspond to pointwise Gaussian upper bounds and generalized volume doubling. Conversely, averaged bounds and volume doubling yield a Faber-Krahn inequality for subsets with prescribed relative measure. The corresponding dimension depends on the scale, but converges to the doubling dimension at infinity. Thus, although the theory is formulated only at large scales, it recovers the classical Gaussian and Faber-Krahn pictures asymptotically.
\\

Classical theories focus on pointwise Gaussian bounds on heat kernels \cite{Varopoulos-85,Davies-87, CarlenKS-87,Davies-93, Grigoryan-94,Sturm-95,Carron-96,Delmotte-99,SaloffC-01,Grigoryan-09,BCS,CKKW-21,GrigoryanHH24,EB25}. In particular, on spaces with suitable cut-off functions, localized upper bounds are characterized by relative Faber-Krahn inequalities on all scales. Discrete graphs however lack good cut-off functions at small scales and require different arguments there. Recently, large-scale versions of the corresponding characterizations have recently been obtained for bounded degree graphs \cite{KellerR26,Rose24}.

For graphs with unbounded degree the situation changes substantially. Pointwise estimates generally retain degree-dependent correction terms \cite{KellerR26,Rose24}. In contrast, the Davies-Gaffney-Grigor'yan estimate gives universal bounds for spatial averages of the heat kernel which do not contain these local degree terms \cite{BauerHuaYau-17}.

The existing results therefore leave a gap between two types of estimates. Pointwise Gaussian bounds on heat kernels admit Faber-Krahn characterizations, but with degree-dependent correction terms; spatially averaged Davies-Gaffney-Grigor'yan estimates avoid these terms, but do not by themselves provide a corresponding Faber-Krahn characterization. 

Our main contribution is to bridge these two types of estimates. We introduce space-time averaged heat semigroup quantities which correspond asymptotically to the pointwise heat kernel. The averaging is not merely a technical relaxation but rather removes the degree-dependent errors in the Gaussian upper bounds while retaining the off-diagonal structure.

The proof of the correspondence with Faber-Krahn inequalities includes two new mechanisms which require genuinely different extensions of the existing heat-kernel machinery. In one direction, the classical Davies' trick to obtain off-diagonal heat kernel bounds based on $\ell^1-\ell^\infty$-semigroup estimates does not yield bounds on the averaged semigroup norms. We replace the trick by a version involving integrated  $\ell^{q/(q-1)}-\ell^q$-norm estimates. In the converse direction, the usual passage from heat-kernel bounds to Faber-Krahn inequalities relies on pointwise control of the heat kernel which are not available here. Instead, we find a new lower bound for Dirichlet eigenvalues in terms of averaged semigroup norms.
\\

In order to compare our results with the existing literature and to illustrate and quantify the problem, we briefly summarize three relevant bounds for the heat kernel $p$  for graphs $b$ with simple weights over the discrete measure space $(X,\Eins)$, i.e., $b(x,y)\in\{0,1\}$, $x,y\in X$. Equip the resulting graph with the intrinsic path metric $\rho$ given by the weights $ w(x,y)=(\deg(x)\vee\deg(y))^{-1/2} $, see Section~\ref{sec:general} for definitions.

As a special case, consider $X=\ZZ$ and the graph defined by the function $b(x,y)=1$ iff $|x-y|=1$, $x,y\in\ZZ$, and zero otherwise. An essentially exact estimate was obtained in \cite[Theorem~3.5]{Pang-93}:
\[
 p_t(x,y)\simeq \frac{1}{\sqrt{\# B_x(\sqrt {t})\# B_y(\sqrt {t})}}\euler^{-t\zeta\left({\rho}/{t}\right)}
\]
for $x,y\in\ZZ$, $t\geq \rho=\rho(x,y)>0$,
where 
\[
\zeta\big(x\big)
=
x\arsinh\left(x\right) 
+1-\sqrt{x^2+1}\]
for $x\geq 0$. 
Note that 
$t\zeta(\rho/t)\sim {\rho^2}/{2t}$
as $t\to \infty$ for all $\rho>0$. Here, $\simeq$ means that the left-hand side can be bounded above and below by the right-hand side up to constants, and $\sim$ that the left-hand side divided by the right-hand side converges to one. 

In the general case, if the vertex degree $\deg$ is uniformly bounded on $X$, an upper bound on the heat kernel similar to the right-hand side can be obtained for large times assuming large-scale Faber-Krahn or Sobolev inequalities \cite{KellerR26,Rose24}. However, if the vertex degree is not uniformly bounded,  the right-hand side has to be modified by a multiplicative term of the form
\[
(\deg(x)\vee\deg(y))^{\frac{1\vee \rho(x,y)}{\sqrt t}}
\]
Conversely, if a heat kernel upper bound of this form and a volume doubling property  hold on large scales, then the dimension of the resulting Faber-Krahn and Sobolev inequalities depend on the maximum of the diagonal of such terms in the ambient ball.

On the other hand, by the Davies-Gaffney-Grigor'yan estimate obtained in \cite{BauerHuaYau-17}, 
for all $A,B\subset X$ and $t\geq 0$, we have
\[
\frac{1}{\# A}\sum_{x\in A}\frac{1}{\# B}\sum_{y\in B}p_t(x,y)\leq \frac1{\sqrt{\# A\cdot \#B}}\euler^{-\Lambda t-t\zeta(\rho(A,B)/t)},
\]
where $\Lambda=\inf\spec(\Delta)$ denotes the spectral bottom of $\Delta$. These bounds do not involve errors depending on $\deg$.

\subsection{Averaged heat kernel and main results}\label{sec:main}

In this section we present global versions of the main results of this article. Let $b$ be a graph over the discrete measure space $(X,m)$ with Dirichlet form $\cE$, associated Laplace operator $\Delta\geq0$, and the heat kernel $p$. Assume that the graph with the intrinsic path metric $\rho$, cf.~ Section~\ref{sec:general}.

For the sake of readability, we identify any set $B\subset X$ with its characteristic function~$\Eins_B$, i.e.,
\[
B:=\Eins_B.
\] Further, for a measurable function $f\colon I\to\RR$ and a set $I\subset \RR$ of finite Lebesgue measure $|I|$, we let
\[
\fint_If:=\frac{1}{|I|}\int_If(x)\drm x.
\]
If $A\colon \ell^q(X,m)\to\ell^r(X,m)$ is a linear operator, its operator norm will be denoted by $\|A\|_{p,q}$, for $q,r\in[1,\infty]$. 
\\

The precise averaged quantity which replaces the pointwise heat kernel is chosen to balance three requirements: it must retain the Gaussian off-diagonal structure, it must average over a spatial scales, and it must converge to the pointwise regime as time increases.

Therefore, we define the (space-time) averaged heat kernel to be the function
\[
\mathcal{P}_r(t,x,y)
:=
\left(
\fint_{I(t,r,x,y)} \| {B_x(r)}P_{\tau} {B_y(r)}\|_{\frac{2q(t)}{2q(t)-1},2q(t)}^{q(t)}\drm \tau
\right)^\frac{1}{q(t)}
\]
where 
\[
I(t,r,x,y)=\left[{t-\frac{r^2}{1\vee t\sigma_{xy}}},{t+\frac{r^2}{1\vee t\sigma_{xy}}}\right]
\]
and 
\[t\sigma_{xy}=\sqrt{\rho(B_x(r),B_y(r))^2S^2+t^2}-t,\qquad q(t)=\frac{1}{21n}\sqrt[2n]{\frac{t}{S^2}}.\]

Our first main theorem shows that relative Faber-Krahn inequalities on large scales are sufficient for averaged Gaussian upper bounds. The estimate is formulated using the distance between balls rather than their centers and contains correction terms depending on the smallest scale at which the Faber-Krahn inequality is assumed. Crucially, these corrections involve ball measures rather than the vertex degree. As a side effect, the theorem yields a corresponding large-scale volume comparison estimate.

\begin{theorem}\label{thm:FK_hk_global}
Let  $C,r_0>0$ and $n>2$ and assume that, for all $B(r)\subset  X$, $r\geq r_0$, and all $U\subset B(r)$, we have the Faber-Krahn inequality
\begin{equation*}
\lambda(U):=
\inf_{0\neq f\in\mathcal{C}_c(U)}
\frac{\cE(f)}{\|f\|_2^2}\geq \frac{C}{r^2}\left(\frac{m(B(r))}{m(U)}\right)^\frac{2}{n}.
\end{equation*}
Then there are constants $C'=C'(C,n)>0$ and $r'=r'(r,S)>0$ such that,
for all $x,y\in X$ and $t\geq r^2\geq r'^2$, we have
\begin{multline*}
\mathcal{P}_r(t,x,y)
 \leq 
\frac{C'}{\sqrt[2q(t)]{m(B_{x}(r))m(B_{y}(r))}}\\
\cdot 
\frac{\left(1\vee S^{-2}\sqrt{\rho(B_x(r),B_y(r))^2S^2+t^2}-t\right)^{\frac{n}{2}}}{\sqrt[\frac{2q(t)}{q(t)-1}]{m(B_x(\sqrt t))m(B_y(\sqrt t))}}
 \euler^{-S^{-2}t\zeta(\rho(B_x(r),B_y(r))S/t)}.
\end{multline*}
Moreover, for all $x\in X$ and $r_0\leq r_1\leq r_2$, we have
\begin{equation*}
m(B_x(r_2))\leq C'\left(\frac{m(B_x(r_2))}{m(B_x(r_0))}\right)^{\sqrt[4n]{{S}/{r_1}}}
\left(\frac{r_2}{r_1}\right)^{n} m(B_x(r_1)).
\end{equation*}
\end{theorem}
These estimates converge, as time increases, to pointwise Gaussian upper bounds while retaining the spatial averaging that removes the degree-dependent errors.

We refer to the upper upper bound on $\mathcal{P}$ obtained in Theorem~\ref{thm:FK_hk_global} as averaged Gaussian upper bounds on the heat kernel. 

The estimate has the expected Gaussian interpretation: we have $2q(t)\to\infty$ and $2q(t)/(q(t)-1)\to 2$ as $t\to\infty$, while the averaging interval converges to the natural time window of length $r^2$. Thus the averaged bound approaches the pointwise Gaussian scale at large times, but without the degree-dependent errors present in the pointwise theory.

\begin{remark}[Behavior of $\mathcal{P}$]Since the function $q$ is strictly monotone increasing and unbounded, for $x,y\in X$ and $r,\tau>0$, the Dunford-Pettis theorem implies
\[
\|{B_x(r)}P_\tau{B_y(r)}\|_{\frac{q(t)}{q(t)-1},q(t)}\to\|{B_x(r)}P_\tau{B_y(r)}\|_{1,\infty}=\sup_{B_x(r)\times B_y(r)}p_\tau 
\]
as $t\to\infty$. Further, if $I$ is an interval of finite length $|I|<\infty$, then we have
\[
\left(\fint_I|f|^{q(t)}\right)^\frac{1}{q(t)}\to \|f\|_{\infty}.
\]
Additionally, $t\sigma_{xy}\to 0$ if $t\to \infty$, such that $I(t,r,x,y)=[t-r^2,t+r^2] $ for $t$ large enough. Hence
\[
\mathcal{P}_r(t,x,y)\sim \sup_{[t-r^2,t+r^2]\times B_x(r)\times 
B_y(r)}p\]
as $t\to \infty$.
\end{remark}

The converse is more subtle. Averaged Gaussian bounds do not directly give a pointwise estimate at arbitrarily small times, such that the classical argument relating the maximum of the heat kernel to the first Dirichlet eigenvalue cannot be applied without modification. We prove instead that averaged semigroup bounds control the first Dirichlet eigenvalue and combine this estimate with large-scale volume doubling.

\begin{theorem}\label{thm:hk_FK_global}
Let $n,r_0,C>0$. Assume that,
for all $x,y\in X$ and $t\geq r^2\geq r_0^2$, we have
\[
\mathcal{P}_r(t,x,y)
 \leq 
C 
\frac{\left(1\vee S^{-2}\sqrt{\rho(B_x(r),B_y(r))^2S^2+t^2}-t\right)^{\frac{n}{2}}}{\sqrt[\frac{2q(t)}{q(t)-1}]{m(B_x(\sqrt t))m(B_y(\sqrt t))}}
 \euler^{-S^{-2}t\zeta(\rho(B_x(r),B_y(r))S/t)}
\]
and  that,
for all $x\in X$ and $r_0\leq r_1\leq r_2$, we have
\begin{equation*}
m(B_x(r_2))\leq C
\left(\frac{r_2}{r_1}\right)^{n} m(B_x(r_1)).
\end{equation*}
Then there is a constant $C'=C'(C,n)>0$ such that, for all $\varepsilon\in[0,1)$, $r\geq 1\vee r_0$, and $U\subset B(r)$ satisfying $m(U)\geq \varepsilon m(B(r))$, we have 
\[
\lambda(U)
\geq 
\frac{C'}{r^{2+\varepsilon}}
\left(
\frac{m(B(r))}{m(U)}\right)^{\frac{2}{N}} \left[1\wedge \frac{1}{m(B(r))}\right] ^{\frac{1}{C'q(r^2)}}
\]
 where 
\[
N(r)= n\cdot \frac{q(r^2)}{q(r^2)-1}+ \frac{1}{\varepsilon}\ln \frac{1}{\varepsilon}\cdot \frac{1}{\ln r}+\frac{1}{q(r^2)-1}\cdot \ln(1\vee m(B(r))).
\]
\end{theorem}

As a consequence, averaged Gaussian bounds imply a relative Faber-Krahn inequality for subsets whose measure is at least a prescribed fraction of the ambient ball. The dimension in this inequality depends on the radius, reflecting the fact that the hypotheses are imposed only at large scales. Nevertheless, for every fixed relative measure $\varepsilon$, this dimension converges to the original dimension as the radius tends to infinity.

The scale-dependent dimension reflects the large-scale nature of the hypotheses rather than a defect of the method. Since the averaged Gaussian estimates are available only beyond a fixed time scale, the eigenvalue argument cannot probe arbitrarily small subsets. The resulting loss in dimension disappears asymptotically.

\begin{remark}Theorem~\ref{thm:hk_FK_global} can be generalized to involve the error in the averaged Gaussian upper bound and doubling property by adjusting the constants in Theorem~\ref{thm:main_hk_FK}. This would lead to more complicated expressions for $C'$ and $N$ involving the measure of the considered ball raised to some decreasing power. Note that for fixed $\varepsilon$, we have $N(r)\to n$ as $r\to\infty$, and that the set of $U\subset B(r)$ satisfying the lower bound on $\lambda(U)$ presented in Theorem~\ref{thm:hk_FK_global} changes as $r$ increases.
\end{remark}

\subsection{About the proof strategy}\label{sec:technique}

Theorems~\ref{thm:FK_hk_global} and~\ref{thm:hk_FK_global} follow from the more general Theorems~\ref{thm:main_FK_hk} and~\ref{thm:main_hk_FK}, respectively. Since there are several tricks which might be of own interest, we summarize the proof ideas shortly.
\\

Theorem~\ref{thm:main_FK_hk} is proven by three ingredients. Relative Faber-Krahn inequalities convert into Sobolev inequalities. Classically, Sobolev inequalities yield $\ell^2-\ell^\infty$-bounds for positive solutions of sandwiched heat equations. In the graph case, these estimates explicitly involve the vertex degree. Our first step is to obtain $\ell^2-\ell^{q(r)}$-estimates for non-negative heat-equation solutions in $r$-balls such that $q(r)\to \infty$ as $r\to\infty$ without such errors. Second, we develop an integrated form of Davies' method which converts these estimates into off-diagonal bounds for averaged semigroup norms. Third, optimizing the corresponding correctors yields the Gaussian decay appearing in Theorem~\ref{thm:main_FK_hk}.

The essential step is the integrated Davies trick. Classically, this trick converts suitable $\ell^2-\ell^\infty$-estimates for sandwiched semigroups into $\ell^1-\ell^\infty$ norm of the heat semigroup, i.e., pointwise estimates. This argument is not applicable in the present setting: we need bounds on a time-averaged $\ell^{q/(q-1)}-\ell^q$-norm of the semigroup, with $q$ depending on the scale. We show that these averaged norms can nevertheless be controlled by $\ell^2-\ell^{q}$-norms of the sandwiched semigroups. Gaussian off-diagonal decay is then obtained by optimizing the correctors with respect to the spatial variable, what requires finding appropriate test functions. We obtain a family of estimates which converges to the pointwise regime as $q\to\infty$, while remaining compatible with the large-scale averaged setting.
\\

The proof of Theorem~\ref{thm:main_hk_FK} starts from a different ingredient. Originally, the Dirichlet eigenvalue is bounded below by the pointwise maximum of the heat kernel, but we assume only an averaged bound on the heat semigroup. Our approach is to  establish a lower bound for the first Dirichlet eigenvalue of an arbitrary subset in terms of averaged norms of the Dirichlet heat semigroup. Resulting bounds are then controlled by averaged semigroup norms of ambient balls. Although this is a  different perspective, the obtained bounds extend the classical estimates asymptotically. 
\\

Combining this estimate with the averaged Gaussian bound and reverse volume doubling gives the required Faber-Krahn inequality. The large-scale restriction prevents the use of arbitrarily small times; this is precisely what leads to the lower measure threshold and the scale-dependent dimension in Theorem~\ref{thm:main_hk_FK}.

\subsection{The set-up}\label{sec:general}

The terminology follows mostly \cite{KellerLW-21}.
Let $X$ be countable and
extend a function
 $m\colon X\to(0,\infty)$ of full support to a measure on $X$ via countable additivity.
A symmetric function $b\colon X\times X\to [0,\infty)$  with 
\[
b(x,x)=0\quad\text{and}\quad \deg(x):=\sum_{y\in X}b(x,y)<\infty
\]
for all $x\in X$ 
is called {graph} over the measure space $(X,m)$. We write $ x\sim y $ whenever $ b(x,y)>0 $ for $ x,y\in X $. A graph is called {locally finite} if the set $\{y\in X\colon b(x,y)>0\}$ is finite for all $x\in X$. It is called connected if for any $x,y\in X$ there exists a finite sequence $x_1,\ldots,x_n\in X$, $n\in\NN$, such that $x_i\sim x_{i+1}$ for all $i\in\{1,\ldots, n-1\}$ and $x_1=x$ and $x_n=y$. 
\eat{\begin{center}
\emph{We assume in the following that all graphs are locally finite and connected.}
\end{center}
}

As usual, denote by $\cC(Y)$ the set of all real-valued functions with support in $Y\subset X$, and by $\cC_c(X)$ the subset of $\cC(X)$ containing all functions of finite support.
For any $p\in[1,\infty]$, let 
\[
\ell^p(X,m):=\{f\in\cC(X)\colon \|f\|_p<\infty\},
\]
 the Banach space of $p$-summable functions,
where for $f\in \cC(X)$
\[
\|f\|_p^p:=\sum_{x\in X}m(x)|f(x)|^p, \qquad p\in[1,\infty),\qquad \|f\|_\infty:=\sup_X|f|.
\] 
The Laplace operator $\Delta\colon \cC(X)\to\cC(X)$ on a graph acts as
\[
\Delta f(x)=\frac1{m(x)}\sum_{y\in X} b(x,y)(f(x)-f(y)), 
\] 
for  $  x\in X  $ and  $ f\in \cC(X) $.
The {associated quadratic form} {$\cE\colon \cC_c(X)\to \RR$} is given by
\[ \cE(f):=\frac{1}{2} \sum_{x,y\in X}b(x,y)|f(x)-f(y)|^2, \quad f\in \cC_c(X).\]
Using Green's formula, we have 
\[\cE(f)=\sum_{x\in X}m(x)f(x)\Delta f(x), \qquad \phi\in \cC_c(X).\]
By local finiteness, $ \Delta $ maps the set of compactly supported functions $ \cC_{c}(X) $ into itself. Hence, the restriction of $ \Delta $ to $ \cC_c(X) $ is symmetric in $ \ell^{2}(X,m)$, the Hilbert space of square integrable functions on $(X,m)$ with respect to the inner product $\langle \cdot,\cdot\rangle_{\ell^2(X,m)}$.  By slight abuse of notation we also denote the closure of $\Delta\restriction_{\cC_c(X)}$ in $\ell^2(X,m)$ by   $\Delta\geq 0$.

The (continuous-time) heat semigroup $(\euler^{-t\Delta})_{t\geq 0}$ acts in $ \ell^{2}(X,m) $ and, for an initial condition $ f\in \ell^{2}(X,m) $, the function $ u=\euler^{-t\Delta}f $ solves the heat equation
\[
\frac{d}{dt}u=-\Delta u, \quad u(0,\cdot)=f.
\] 
Moreover, the heat semigroup has an integral kernel $p\colon [0,\infty)\times X\times X\to[0,\infty)$, called the \emph{heat kernel},
satisfying
\[
\euler^{-t\Delta}f(x)=\sum_{y\in X} m(x) p_t(x,y)f(y)
\]
for all $ x\in X, t\geq 0, f\in \ell^2(X,m) $. 

In general it is not possible to provide an exact description of the heat kernel.
Of fundamental interest is the behavior of the heat kernel $p$ in terms of geometric properties of the graph. 
\\

A vital tool for dealing with unbounded Laplacians on graphs are intrinsic metrics, cf.~\cite{Davies-93a,Folz-11,GrigoryanHuangMasamune-12,  BauerKW-15,Keller-15,KellerLW-21,KellerRose-22b,Rose-26}.
An {intrinsic metric}  with respect to $b$ over $(X,m)$  is a non-trivial pseudo-metric $\rho\colon X\times X\to[0,\infty)$  such that 
\[
\sum_{y\in X} b(x,y)\rho^2(x,y)\leq m(x),
\]
for all  $ x\in X $. 
A pseudo metric $ \rho  $ is called a {path metric} with {respect to the graph}  if there is $ w: X\times X\to [0,\infty] $ such that for all $ x,y\in X $
\begin{align*}
	\rho (x,y)=\inf_{x=x_{0}\sim \ldots\sim x_{n}=y}\sum_{j=1}^{n}w(x_{j-1},x_{j})
\end{align*}
{and $w(x,y)<\infty $ iff $ x\sim y $.}
Observe that the choice $ w(x,y)=(\Deg(x)\vee\Deg(y))^{-1/2} $ for $ x\sim y $ and $ w(x,y) =\infty$ otherwise yields an intrinsic path metric.

As usual, we let $$ {B(r)=} B_x(r)=\{y\in X\mid \rho(x,y)\leq r\},$$ $r\geq 0 $, $x\in X$. 

\begin{center}{\em We assume that the graph is equipped with an intrinsic metric which is a path metric for which the  distance balls are compact and we have the \emph{finite jump size} condition}  $${ S:=\sup\{\rho(x,y)\colon x,y\in X, b(x,y)>0\}<\infty }.$$
\end{center}

As a consequence our graphs are {locally finite} and {connected}. Indeed, local finiteness follows from finite balls and finite jump size, while connectedness follows from the fact that $ \rho $ is a path metric taking values in $ (0,\infty) $, cf. \cite{KellerLW-21}.
Moreover, the assumptions  above yield that the metric space $ (X,\rho) $ is complete and geodesic, i.e., for any two vertices $ x,y\in X $ there is a path $x= x_{0}\sim\ldots \sim x_{n}=y $ such that $ \rho(x,y)=\rho(x,x_{j})+\rho(x_{j},y) $ for all $ j=0,\ldots,n $, see \cite[Chapter~11.2]{KellerLW-21}.
\eat{
Given this data, a universal heat kernel bound is nowadays referred to as the Davies-Gaffney-Grigor'yan estimate which holds on any graph equipped with an intrinsic metric. 

\begin{theorem}[{\cite{BauerHuaYau-17}}]\label{thm:bhy}For all $A,B\subset X$ and $t\geq 0$, we have
\[
\frac{1}{m(A)}\sum_{x\in A}m(x)\frac{1}{m(B)}\sum_{y\in B}m(y)p_t(x,y)\leq \frac1{\sqrt{m(A) m(B)}}\euler^{-\Lambda t-t\zeta(\rho(A,B)/t)},
\]
where $\Lambda=\inf\spec(\Delta)$ denotes the spectral bottom of $\Delta$.
\end{theorem}
\eat{\begin{definition}
The {\em Sobolev inequality} $S_\phi(n,r_1,r_2)$ holds in $B$, if for all $x\in B$, $r\in[r_1,r_2]$ {and $u\in\cC_{c}(B_x(r))$}, we have
\begin{equation*}
\frac{m(B_x(r))^{\frac{2}{n}}}{\phi r^2}
\Vert u\Vert_{\frac{2n}{n-2}}^2
\leq \cE(u)+\frac{1}{r^2}\Vert u\Vert_2^2.
\end{equation*}
\end{definition}
}

}
\eat{
\begin{definition}
The {\em relative Faber-Krahn inequality} $FK_a(n,r_1,r_2)$ holds in $B$, if for all $x\in B$, $r\in[r_1,r_2]$ and $U\subset B_x(r) $, we have
\begin{equation*}
\lambda(U):=\inf_{u\in\cC_c(U),\|u\|_2= 1}\cE(u)\geq \frac{a}{r^2}\left(\frac{m(B_x(r))}{m(U)}\right)^\frac{2}{n}.
\end{equation*}
\end{definition}
}

\eat{
\subsection{Gaussian upper heat kernel bounds for the counting measure}\label{sec:counting}
\eat{In order to illustrate our results we consider in this section the special case of the counting measure and compare it with the existing literature. 
}
Let $m=\Eins$ be the counting measure on $X$ assigning to each subset of $X$ the number of contained vertices, and $b$ be a graph over $(X,\Eins)$ with simple weights, i.e., $b(x,y)\in\{0,1\}$, $x,y\in X$.  Equip the resulting graph with the intrinsic path metric $\rho$ given by the weights $ w(x,y)=(\deg(x)\vee\deg(y))^{-1/2} $.

For the graph constructed from the set of integers, i.e., $X=\ZZ$, and the graph defined by the function $b(x,y)=1$ iff $|x-y|=1$, $x,y\in\ZZ$, and zero otherwise, an essentially exact estimate for the heat kernel was obtained by Pang \cite[Theorem~3.5]{Pang-93} :
\[
 p_t(x,y)\simeq \frac{1}{\sqrt {t\vee \rho}}\euler^{-t\zeta\left({\rho}/{t}\right)}
\]
for $x,y\in\ZZ$, $\rho=\rho(x,y),t>0$,
where 
\[
\zeta\big(x\big)
=
x\arsinh\left(x\right) 
+1-\sqrt{x^2+1}\]
for $x\geq 0$. 
Here, $\simeq$ means that the left-hand side can be bounded by the right-hand side up to constants. Note that 
\[
t\zeta(\rho/t)\sim \frac{\rho^2}{2t}
\]
as $t\to \infty$ for all $\rho>0$. Here, $\sim$ means that the left-hand side divided by the right-hand side converges to one. Note that, in this setting, $\sqrt2\cdot \rho$ is the combinatorial distance. 

For large times, the pointwise behavior of the heat kernel on the integers is comparable to the Gauss-Weierstrass function on the real line. 
In the following, upper bounds on the heat kernel of a graph which are comparable to the heat kernel estimate on the integers are referred to as Gaussian bounds.
\eat{ The heat kernel $p$ associated with $\Delta$ is defined as the minimal positive solution of the $\ell^2(X,m)$-Cauchy problem
\[
\partial_t u=-\Delta u, \qquad u(0)=\Eins_x, \qquad x\in X.
\]

For precise definitions, we refer the reader to  Section~\ref{sec:graphs}.
}



Next, we recall a characterization of pointwise Gaussian upper bounds for the heat kernel in terms of relative Faber-Krahn inequalities obtained in \cite{Rose24}. It constitutes a variant of a classical result on Riemannian manifolds by Grigor'yan \cite{Grigoryan-94} and is inspired by ideas from \cite{KellerRose-22a, KellerR26} dealing with local Sobolev inequalities.

Let $r_0,t_0\geq 0$ and $n>0$.  

We say that the heat kernel $p$ satisfies Gaussian upper bounds  if there is a constant $C>0$ and a function $\epsilon \colon [r_0,\infty)\times[0,\infty) \to[0,\infty)$, with $\epsilon(t,\rho) \to0$ if $t\to\infty$ and monotone decreasing for all $\rho\in[0,\infty)$, such that, for all $t\geq t_0$ and all $x,y\in X$, the heat kernel has the upper bound
\begin{equation}\tag{$G_{c}(\epsilon, t_0)$}
p(t, x,y)
\leq C(\deg(x)\vee \deg(y))
^{\epsilon(t,\rho(x,y))}
\frac{\left(1+\sqrt{t^2+\rho(x,y)^2}-t\right)^{\frac{n}{2}}}
 {\sqrt{|B_{x}(\sqrt {t})||B_{y}(\sqrt {t})|}}
\euler^{-t\zeta\left(\rho(x,y)/t\right)}.
\end{equation}
These bounds are characterized in terms of relative Faber-Krahn inequalities for a dimension function $n'\colon \{B_x(r)\subset X\colon x\in X, r\ge0\}\to(0,\infty)$: there exists a constant $r_0\geq 0$ and a function $C\colon(0,\infty)\to(0,\infty)$ such that,  for all $B(r)\subset  X$, $r\geq r_0$, we have
\begin{equation}\tag{$FK_{c}(r_0)$}
\lambda(U):=
\inf_{0\neq f\in\mathcal{C}_c(U)}
\frac{\cE(f)}{\|f\|_2^2}\geq \frac{C^{n'(B(r))}}{r^2}\left(\frac{|B(r)|}{|U|}\right)^\frac{2}{n'(B(r))}
\end{equation}
\eat{These bounds are characterized in terms of local Sobolev inequalities for a dimension function $n'\colon \{B\subset X\}\to(2,\infty)$: there exists a  {constant} $C>0$ such that  for all $B(r)\subset  X$, $r\geq r_0$ and $u\in\cC_c(B(r))$, we have
\begin{equation}\tag{$S_{count}(r_0)$}
C\frac{|B(r)|^{\frac{2}{n'(B(r))}}}{r^2}
\Vert u\Vert_{\frac{2n'(B(r))}{n'(B(r))-2}}^2
\leq \cE(u)+\frac{1}{r^2}\Vert u\Vert_2^2.
\end{equation}
}
The characterization involves  the following two regularity properties of the measure.

The {volume doubling} property is satisfied in $X$ if there is a constant $C>0$ and a function $\epsilon \colon [r_0,\infty) \to[0,\infty)$, with $\epsilon(t) \to0$ if $t\to\infty$ and monotone decreasing, such that, for all $x\in X$ and $r_0\leq r_1\leq r_2$, we have
\begin{equation}\tag{$V_{c}(\epsilon, r_0)$}
|B_x(r_2)|\leq C\deg(x)
^{\epsilon(r_1)}\left(\frac{r_2}{r_1}\right)^{n} |B_x(r_1)|.
\end{equation}

 The local regularity property is satisfied in $ X $ if there is a constant $C>0$ such that for all $x\in X$ and $r\geq r_0$, we have 
\begin{equation}\tag{$L(r_0)$}
\frac{|B_x(r)|}{r^n}
\leq C\deg(x)^{\frac{n}2}.
\end{equation}

\begin{theorem}[{\cite{Rose24}}]\label{thm:r24}
\begin{enumerate}[(i)]
\item Let $r_0,n > 0$ be constants such that ($FK_{c}(r_0))$) with dimension $n$ holds in $X$. Then there exists a function $\epsilon=\epsilon(n) \colon [r_0,\infty)\times [0,\infty) \to[0,\infty)$, with $\epsilon(t,x) \to0$ if $t\to\infty$ monotone decreasing for all $x\in [0,\infty)$, and a constant
$t_0 = t_0(r_0) >0$ such that ($G_{c}(\epsilon,t_0)$), ($V_{c}(\epsilon(\cdot, 0),r_0)$) and ($L(r_0)$) hold in $X$.
\item Conversely, let $r_0,t_0, n>0$ be constants and $\epsilon \colon [r_0,\infty)\times [0,\infty) \to[0,\infty)$ a function with $\epsilon(t,x) \to0$ if $t\to\infty$ and monotone decreasing  for all $x\in [0,\infty)$, such that properties
($G_{c}(\epsilon,t_0)$), ($V_{c}(\epsilon(\cdot,0),r_0)$) and ($L(r_0)$) hold in $X$. Then there is $r_1 = r_1(n,r_0,t_0) > 0$
such that ($FK_{c}(r_1)$) holds with  
dimension function 
\[
n'(B_x(r))=n(1+\epsilon(r))\left(1+\epsilon(r)+\frac1{2\ln r}{\ln \left(1\vee \max\limits_{B_x(cr\max\limits_{B_x(r)}\deg^{c/r})}\deg\right) }\right).
\]
\end{enumerate}
\end{theorem}

\eat{\begin{theorem}[{\cite{KellerR26}}]\label{thm:kr26}
If $r_0>0$ and $n>2$ are constants and $S_{count}(r_0)$ with dimension $n$ holds in $X$, then there exists $t_0=t_0(r_0)>0$ such that $G_{count}(t_0)$, $V_{count}(r_0)$ and $L(r_0)$ hold in $X$. Conversely, if $n{>2}$, $r_0,t_0>0$ are constants such that $G_{count}(t_0)$, $V_{count}(r_0)$ and $L(r_0)$ hold in $X$, then there exists $r_1=r_1(r_0,t_0)>0$ such that  $S_{count}(r_1)$ holds with dimension function $n'$ given by 
\[
n_x'(r)
=
n \left(1+284n\frac{ \ln(\ln r)  }{\ln(r)  } \right)\left(1+\frac{1}{2}\cdot\frac{\ln d_{x}(r)}{  \ln r}\right),\]
where $d_x(r)=\max_{B_x(r)}(1\vee\deg)$.
\end{theorem}

}

\begin{remark}[Uniformly bounded degree]
Note that Theorem~\ref{thm:r24} simplifies if the vertex degree is uniformly upper bounded. In this case, the involved constants depend on the uniform bound and the dimension is just the doubling dimension. Moreover, the condition $L(r_0)$ can be dropped, cf.~\cite{Rose24}. In particular, this simplified version with $r_0=0$ qualitatively compares directly to the settings of Riemannian manifolds .
\end{remark}
\eat{
To summarize, Theorem~\ref{thm:bhy} is a universal upper bound depending only on the metric structure induced by $\rho$ on the graph. In contrast, 
Theorem~\ref{thm:r24} is a characterization of pointwise Gaussian upper heat kernel bounds where involved estimates or constants depend explicitly on the vertex degree. The errors involve the vertex degree since there are no cut-off functions on arbitrarily small scales. 
\\
}

\eat{
Motivated by above considerations, we present Gaussian upper bounds on space-time averages of the heat semigroup rather than the classical pointwise Gaussian upper bounds for the heat kernel.
This generalization has two advantages. First, the obtained estimates do not involve errors depending on the vertex degree. Second, the obtained weighted heat kernel bounds interpolate between the Davies-Gaffney-Grigor'yan lemma and pointwise Gaussian upper bounds. Our second aim is to present relative Faber-Krahn inequalities if only such space-time averages of the heat semigroup and a generalized volume comparison hold on a graph. In this case, the Faber-Krahn inequality only holds for sets whose measure is a given fraction  of the measure of a ball.
\\
}

}

\section{From Faber-Krahn to averaged heat kernel bounds}\label{sec:FK_hk}
This section is devoted to the derivation of averaged Gaussian bounds on the heat kernel. Section~\ref{sec:FKVM} provides consequences of the Faber-Krahn inequality regarding volume doubling and $\ell^2$-$\ell^p$-bounds for positive solutions of sandwiched heat equations and integrated operator norms of sandwiched semigroups. In Section~\ref{sec:davies}, we obtain an integrated version of Davies' trick to obtain off-diagonal averaged heat kernel bounds. Section~\ref{sec:proofmain} puts the $\ell^2-\ell^p$-bounds and the integrated Davies' trick together to prove off-diagonal averaged Gaussian upper heat kernel bounds. 

\subsection{Faber-Krahn, volume doubling, and $\ell^2$-$\ell^p$-bounds}\label{sec:FKVM}

This section is devoted to the study of consequences of the Faber-Krahn inequality regarding volume doubling and $\ell^2$-$\ell^p$-bounds for positive solutions of sandwiched heat equations. It is classical that Faber-Krahn inequalities yield Sobolev inequalities, cf.~Proposition~\ref{prop:FKS}. This allows to  modify recent results obtained in \cite{KellerR26} in order to fit our purposes.
\\

We start with the following relation between Faber-Krahn and Sobolev inequalities.

\begin{proposition}\label{prop:FKS}Assume $c>0$, $n>2$, and $B\subset X$ is such that, for all $U\subset B$, we have 
\[
\lambda(U)\geq c\ m(U)^{-\frac{2}{n}}.
\]
Then, for all $f\in\cC_c(B)$, we have 
\[
2^{-\frac{2n+4}{2n-2}}c\|f\|_{\frac{2n}{n-2}}^2\leq \cE(f).
\]
\end{proposition}
The proof of Proposition~\ref{prop:FKS} is standard and can be found in many excellent textbooks, see, e.g., \cite[Theorem~3.14]{Barlow-book}.

The following is a slight adaption of some results obtained in \cite{KellerR26}.

\begin{proposition}[Non-collapsing]\label{lem:asnoncoll}Let $x\in X$,  $R \ge 4R_1{\ge 4S}$, 
$n>2$ and $C>0$ constants. Assume  that for all $f\in \cC_{c}(B_x(R))$ we have 
\[
\Vert f\Vert_{\frac{2n}{n-2}}^2\leq C\left(\cE(f)+\frac{1}{R^2}\Vert f\Vert_2^2\right).
\]
Then we have for all {$r\in[2S,R]$}
{
\[
 2^{-{3}n^2}\left(\frac{1}{ C}\right)^{\frac{n}{2}}
\left[1\wedge C^{\frac{n}{2}}\frac{m(B_x(R_1))}{r^{n}}\right]^{2^{2}\sqrt[n]{{S}/{r}}}
\leq \frac{m(B_x(r))}{r^{n}}.
\]}
\end{proposition}

\begin{proof}
The proof is the same as the one of \cite[Lemma~xx]{KellerR26} with the only difference that we replace the estimate $m(r)\geq m(x)$ by the estimate $m(r)=m(B_x(r))\geq m(B_x(R_1))=m(R_1)$ in the notation of the cited lemma.
\eat{
Lemma~\ref{lem:StoN} yields for all $f\in \cC_{c}(B_x(R))$ the Nash inequality
\[
\Vert f\Vert_2^{2+4/n}\leq C \left(\Vert |\nabla f|\Vert_2^2+\frac{1}{R^2}\Vert f\Vert_2^2\right)\Vert f\Vert_1^{4/n}.
\]
We apply this to special cut-off functions.
For  {$r\in[2S,R]$}, choose
 
{\[
f_r(y):=
\begin{cases}
	{r}-\rho(x,y),& y\in B_{x}(r/2),\\
\left(\frac{3r}2-2\rho(x,y)\right)_+,& 
\mbox{else},
\end{cases}
\]
}
which satisfies  {$\supp f_r\subset B_x(3r/4)\subset B_x(r)$}, 
 {\[
\Vert f_r\Vert_1\leq {r} m(r), \quad 
 \quad \frac{r}{2}m(r/2)^{1/2}\leq \Vert f_r\Vert_2\leq {r} m(r)^{1/2},
\]}
where we used the notation $ m(r)=m(B_{x}(r)) $.
Since  {$ r\ge 4S $}, we have  
 { $\supp\vert \nabla f_r\vert\subset B_x(3r/4+S)\subset B_x(r)$}. Thus, since $\vert \nabla f_r\vert\leq {2}$, 
\[
\Vert \vert \nabla f_r\vert \Vert_2
\leq  {2} m(r)^{1/2}.
\]
Therefore, the Nash inequality applied to $f_r$ yields 
 {
\begin{align*}
	&\left(\frac{r}2 m\left({r}/2\right)^{\frac12}\right)^{2+\frac4n}
	\leq 
	C \left(4m(r)+\frac{({r})^2}{R^2}m\left({r}\right)\right)
	 {r}^{\frac4n}\left( m\left({r}\right)\right)^{\frac4n}
	\leq 8 C \ m(r)^{1+\frac4n}
	r^{\frac4n},
\end{align*}	
}
where we used  {$ r\le  R $} and the monotonicity of the measure in the last line. If we 
put 
\[\alpha=1+\frac2n,\quad \beta=1+\frac4n,\quad\text{and}\quad q=\frac\alpha\beta{=\frac{n+2}{n+4}},
\]
this is equivalent to
 {\begin{align*}
\left(\frac{r^2}{ 2^{2\alpha+3}C}\right)^{\frac{1}{\beta}}  m\left({r}/2\right)^{q}
\leq   m(r).
\end{align*}}
Iterating the above inequality we obtain
 {\begin{align*}
\left(\frac{r^2}{ 2^{2\alpha+3}C}\right)^{\frac{1}{\beta}\sum_{i=0}^{k-1}q^i}\left(\frac12\right)^{\frac{2}{\beta}\sum_{i=1}^{k-1}iq^i} 
m\left({r}/{2^k}\right)^{q^k}
\leq   m(r).
\end{align*}}
This iteration procedure yields a non-trivial lower bound on $m(r)$ as along as we have  { $r/2^{k-1}\geq 4S$}, i.e., if we choose
\[
k\leq
\left\lfloor  \log_2\frac{r}{2 S}\right\rfloor =:\eta(r).
\]
We are left with the sums in the exponents and start with the calculation of the first from the left. Set $\theta:=q^{\eta}$. Clearly, geometric summation gives $\sum_{i=0}^{ { \eta}-1}q^i=\frac{1-q^{ {\eta}}}{1-q}=\frac{1-{ {\theta}}}{1-q}$. Next, since $q\in(0,1)$ we have 
\[
\sum_{{i}=1}^{k-1}iq^{i}\leq \sum_{i=0}^\infty iq^i=\frac{q}{(q-1)^2}.
\]
Hence, using $m(r)=m(B_x(r))\geq m(B_x(R_1))=m(R_1)$ for all $r\geq 0$,
{\begin{align*}
 m(B_x(r))\ge\left(\frac{r^2}{  2^{2\alpha+3} C }\right)^{\frac{1}{\beta}\frac{1- {\theta}}{1-q}}  
\left({\frac 12}\right)^{\frac{2}{\beta}\frac{q}{(1-q)^2}}
m(R_1)^{ {\theta}}
=\left(\frac{r^2}{ C }\right)^{\frac{n}{2}(1- {\theta})} 
 {\left(\frac12\right)^{(\frac52n+2)(1-\theta)+\frac{n^2}{2}+1}}
m(R_1)^{ {\theta}}
\end{align*}}
where, by noting $q=\alpha/\beta$, $\alpha=1+2/n$ and $\beta=1+4/n$, we used the trivial identities $\beta(1-q)=\beta(1-\alpha/\beta)=\beta-\alpha=2/n$ and 
$q/(1-q)=\alpha/(\beta-\alpha)=n/2+1$. 
Since $ {\theta}=q^{ {\eta}}\in(0,1)$, we get for the exponent of  $1/2$  using $ n>2 $
\[
\left(\frac52n+2\right)(1-\theta)+\frac{n^2}{2}+n
\leq \frac{n^2}{2}+\frac{7}{2}n+2 \leq  3 n^2.
\]
This leads to the estimate
{\[
 2^{-{3}n^2}\left(\frac{1}{ C}\right)^{\frac{n}{2}}
\Phi^{\theta }
\leq \frac{m(B_x(r))}{r^{n}}.
\]}
where $\Phi= C^{{n}/{2}}{m(B_x(R_1))}/{r^{n}} $,
$ \theta =q^{\eta(r)}$, $q=(n+2)/(n+4)\in(0,1)$, and {$ \eta(r)=\lfloor  \log_2 (r/2S) \rfloor$}. In order to obtain the claim, we estimate $\Phi\geq 1\wedge \Phi$ and bound $\theta$ from above. From $\lfloor x\rfloor\geq x-1$ for all $x\in\RR$ we infer
{$$
\eta\geq \log_2{r/2S}-1\geq \log_2{r/4S}.
$$} 
{By using the elementary inequality $ \ln (1+x)\ge 2x/(2+x) $  for $ x\ge 0 $, we obtain the inequality
$ n\ln ((n+4)/(n+2))\ge 2n/(n+3) $. %
Furthermore, $ 2n/(n+3)\ge \ln 2 $ for $ n>2 $. Thus, we have $$  \log_2({1/q})=\log_2({(n+4)/(n+2)})\ge 1/n . $$
}
Therefore, since $S/r\leq1$,
 {\[
\theta=q^{\eta}\leq {q}^{\log_2(\frac{r}{4S})}=4^{\log_2({1/q})}\left(\frac{S}{r}\right)^{\log_2({1/q})}\leq 2^2\left(\frac{S}{r}\right)^{\frac{1}{n}}.
\]}
Applying this estimate leads to the claim.
}
\end{proof}

\begin{proposition}[Volume doubling]\label{prop:doubling} Let $x\in X$,   {$R_2\geq R_1\geq 4S$}, $n>2$ and $\phi\geq 1$. If for all $r\in[R_1,R_2]$ {and $u\in\cC_{c}(B_x(r))$}, we have
\begin{equation*}
\frac{m(B_x(r))^{\frac{2}{n}}}{\phi r^2}
\Vert u\Vert_{\frac{2n}{n-2}}^2
\leq \cE(u)+\frac{1}{r^2}\Vert u\Vert_2^2,
\end{equation*}
then we have
{\[
\frac{m(B_x(R))}{R^{n}}
\leq {A(n)} \Phi_x(n,r,R)\ \frac{m(B_x(r))}{r^{n}}, \quad R_1\leq r\leq R\leq R_2,
\]
}
where
{
\[
\Phi_x(n,r,R)=\left[1\vee  r^n
{\frac{m(B_x(R))}{m(B_x(R_1))R^n}}
\right]^{ 4\sqrt[n]{S/r}},\quad\mbox{and}\quad A(n)=2^{ 8n^2}\phi^{\frac{n}{2}}(1\vee S)^{4n}.
\]
}
In particular, if $R_2= 2R_1$, we obtain 
\begin{align*}
	m(B_x(R_2))
\leq {A'(n)} \Phi_x(n,R_2) m(B_x(R_2/2)),\qquad A'(n)=2^{ 9n^2}\phi^{\frac{n}{2}}(1\vee S)^{4n}
\end{align*}
where $\Phi_x(n,R_2)=\left[
{\frac{m(B_x(R_2))}{m(B_x(R_2/2))}}
\right]^{ 6\sqrt[n]{S/R_2}}\geq \Phi_x(n,R_2/2,R_2)$. 
\end{proposition}
\begin{proof}The proof is the same as the proof of \cite[Theorem~2.5]{KellerR26} with the only differences that we use Lemma~\ref{lem:asnoncoll} instead of \cite[Lemma~2.3]{KellerR26}, and that we do not use the local regularity condition.
For the in particular statement, we use that, since $n>2$, we have $4\cdot 2^{1/n}\leq 4\sqrt{2}\leq 4\cdot 3/2=6$, such that, for all $a\geq 1$, we get $a^{4\sqrt[n]{S/(R_2/2)}}=a^{42^{1/n}\sqrt[n]{S/R_2}}\leq a^{6\sqrt[n]{S/R_2}}$. 
This finishes the proof.
\eat{Let $R\in[R_1,R_2]$, and  {$w:=2^{-2}\sqrt[n]{r/S}$}. The non-collapsing lemma,
Lemma~\ref{lem:asnoncoll}, yields with $C=\phi\frac{R^{2}}{m(B(R))^\frac{2}{n}}$ for all  {$r\in [ 4S,R]$}
{
\[
\frac{1}{2^{ {3}n^2}\phi^\frac{n}{2}}\frac{m(B_x(R))}{R^{n}}
\left[1\wedge  \phi^{\frac{n}{2}}\frac{R^{n}}{m(B_x(R))}\frac{m(B_x(R_1))}{r^{n}}\right] ^{ \frac{1}{w}}
\leq \frac{m(B_x(r))}{r^{n} }.
\]
}
Division by the first and third factor  yields
\begin{multline*}
\frac{m(B_x(R))}{R^{n}}
\leq 
2^{  {3}n^2}
\phi^{\frac{n}{2}}
\left[1\vee \left(\frac{r}{R}\right)^{n} {\frac{m(B_x(R))}{\phi^{\frac{n}{2}}m(B_x(R_1))}}\right]^{ \frac{1}{w}}\frac{m(B_x(r))}{r^{n}}
\\
\leq 
2^{  {3}n^2}
\phi^{\frac{n}{2}}(1\vee r^\frac{n}{w})
\left[1\vee  
{\frac{m(B_x(R))}{m(B_x(R_1))R^n}}
\right]^{ \frac{1}{w}}\frac{m(B_x(r))}{r^{n}},
\end{multline*}
where we used $\phi\geq 1$.
Now, note
{
\[
r^{\frac{n}{w}}=S^{2^2n\sqrt[n]{S/r}}h(r/S)^{2^2n}
\] 
with $$ h(x)=x^{\sqrt[n]{1/x}}=\exp\left(\frac{\ln x}{x^{\frac{1}{n}}}\right).$$
Differentiating with respect to $ x $ yields that the maximum is attained at $ x= \euler^{n}$ and hence we can estimate $ h(r/S)^{2^2n}\leq e^{2^2n^2/e}\leq 2^{3n^2} $ since $ e^{{4/e}}\leq 2^3 $.} 
The second claim is an application of the first one by observing 
\[
\left[1\vee \frac{1}{R_1^n}\right]^{4\sqrt[n]{S/R_1}}
\leq \left[1\vee \frac{(1\vee S)^n}{R_1^n}\right]^{4\sqrt[n]{(1\vee S)/R_1}}
=1\vee h((1\vee S)/R_1)
\] 
where $h(x)=x^{4n/\sqrt[n]{x}}\leq h(e^n)=e^{4n^2/e}\leq 2^{4n^2}$.
This finishes the proof.
}
\end{proof}

For a $\rho$-Lipschitz function $\omega$ we define
\[
\Delta_\omega v(x):=\euler^{\omega(x)}\Delta(\euler^{-\omega}v)(x)\quad\text{and}\quad h(\omega):=\sup_{x\in X} \sum_{y\in X} \frac{b(x,y)}{m(x)}\vert (\euler^{\omega(x)}-\euler^{\omega(y)})(\euler^{-\omega(x)}-\euler^{-\omega(y)})\vert.
\]
In this paragraph, $v\geq 0$ denotes a non-negative subsolution of $\partial_t+\Delta_\omega$ on $[0,\infty)\times X$, i.e.,
\[
\frac{d}{dt}v+\Delta_\omega v\leq 0.
\] 

The following is from \cite[Theorem~2.7]{KellerRose-22a}.

\begin{proposition}[Moser in time and space, cf.~{\cite[Theorem~2.7]{KellerRose-22a}}]\label{prop:moser}
Let $ x\in X $, $T\in\RR$, 
$n>2$, $\delta\in(0,1]$, $r\geq 128 S$ and constants $\phi,\Phi\geq 1$.
Assume for all $s\in[r/2,r]$ {and $u\in\cC_{c}(B_x(s))$}, we have
\begin{equation*}
\frac{m(B_x(s))^{\frac{2}{n}}}{\phi s^2}
\Vert u\Vert_{\frac{2n}{n-2}}^2
\leq \cE(u)+\frac{1}{s^2}\Vert u\Vert_2^2,
\end{equation*} and the doubling property
\[
m(B_x(r))\leq \Phi\  m(B_x(r/2)).
\]
 For all non-negative $\Delta_\omega$-subsolutions $v\geq 0$ on $[T-r^2,T+r^2]\times B_{x}(r)$, we have
 \[
\left(\frac{1}{
m({B_{x}(r/2)})}\fint\limits_{ T-\delta (r/2)^2}^{T+\delta (r/2)^2}\sum_{B_{x}(r/2)} mv_t^{2w(r)}\drm t\right)^{\frac{1}{w(r)}}
\\
\leq 
 \frac{C_{n,\phi,\Phi}(1+\delta r^2h(\omega))^{\frac{n}{2}+1}  }{\delta^{\frac{n}{2}}
 m({B_{x}(r)})}
\fint\limits_{T- \delta r^2}^{T+\delta r^2}\sum_{B_{x}(r)} mv_t^{2}\drm t,
\]
where  
\[
C_{n,\phi,\Phi}:=2^{ {110}n^2}(\phi\Phi)^{2n}\qquad\text{and}\qquad w(r)={\left(\frac{r}{8S}\right)^{\frac{1}{n}}}.
\]
\end{proposition}

With the help of the $\ell^2-\ell^p$-bounds obtained in Proposition~\ref{prop:moser}, we derive bounds for norms of sandwiched heat semigroups.

For the semigroup $ (P_{t})_{t\ge0} $ of the Laplacian $ \Delta $ on $ \ell^{2}(X,m) $ and $\omega\in \ell^\infty(X)$, we let
\[
P_t^{\omega}:=\euler^{\omega}P_t\euler^{-\omega}.
\]
Then, $(P_t^{\omega})_{t\geq 0}$ is also a semigroup of bounded operators on $\ell^2(X,m)$. 
For $f\in \ell^2(X,m)$, the map $t\mapsto P_t^{\omega} f$ solves the $\ell^2$-Cauchy problem for the operator $\euler^\omega \Delta\euler^{-\omega}$, i.e.,
\[
\frac{d}{d t} P_t^{\omega} f=-\euler^\omega \Delta\euler^{-\omega}P_t^{\omega} f,
\]
since $P_tf$ is the solution of the $\ell^2$-Cauchy problem for $\Delta$. Proposition~\ref{prop:moser} yields the following for the operator norm of the semigroup $(P^\omega_t)_{t\geq 0}$.

\begin{theorem}\label{thm:subsolution}
Let $ x\in X $, $T\in\RR$, 
$n>2$, $\delta\in(0,1]$, $r\geq r_1/2\geq 128 S$ and constants $\phi\geq 1$.
Assume $S_{\phi}(n,r/2,r)$ in $ x$. Then, for all $\omega\in\ell^\infty(X,m)$, we have 
 \begin{multline*}
\left(\frac{1}{m({B_{x}(r_1)})}\fint\limits_{ T-\delta r_1^2}^{T+\delta r_1^2}\|{B_{x}(r_1)} P_t^\omega\|_{2,2q}^{2q}\drm t\right)^{\frac{1}{q}}
\\
\leq \left(\frac{m({B_{x}(r)})}{m(B_x(r_1))}
\right)^{\frac{3}{q}}\frac{C_{n,\phi}(1+\delta r^2h(\omega))^{\frac{n}{2}+1}  }{\delta^{\frac{n}{2}} m({B_{x}(r)})}
\fint\limits_{T- \delta r_2^2}^{T+\delta r_2^2}\|{B_{x}(r)} P_t^\omega\|_{2,2}^{2}\drm t.
\end{multline*}
where 
\[
C_{n,\phi}:=2\cdot2^{ {111}n^2}(\phi{A'(n)} )^{2n}\qquad\text{and}\qquad q=q(r^2)=\frac{1}{21n}\sqrt[2n]{\frac{r^2}{S^2}}
\]
\end{theorem}

\begin{proof}Abbreviate $q:=q(r^2)$ and $w:=w(r)$. In the first step, we derive the $\ell^2$-$\ell^{q}$-estimate for general non-negative $\omega$-solutions $v$. The second step is then devoted to the derivation of the claimed $\ell^2$-$\ell^{q}$-estimate involving the operator norms.

Since $n>2$ we have $3q\leq w$ and hence $\vertiii{f}_{L^{q}(Q)}\leq \vertiii{f}_{L^{w}(Q)}$. We have
\begin{multline*}
\left(\frac{1}{m({B_{x}(r_1)})}\fint\limits_{ T-\delta r_1^2}^{T+\delta r_1^2}\sum_{B_{x}(r_1)} mv_t^{2w}\drm t\right)^{\frac{1}{w}}
\\=
\left(\frac{m({B_{x}(r/2)})}{m(B_x(r_1))}\frac{r^2}{r_1^2}\right)^{\frac{1}{w}}\left(\frac{1}{m({B_{x}(r/2)})}\frac{1}{2\delta (r/2)^2}\int\limits_{ T-\delta r_1^2}^{T+\delta r_1^2}\sum_{B_{x}(r_1)} mv_t^{2w}\drm t\right)^{\frac{1}{w}}
\\
\leq 
\left(\frac{m({B_{x}(r)})}{m(B_x(r_1))}
\right)^{\frac{1}{w}}
\left(\frac{r^2}{r_1^2}\right)^{\frac{1}{w}}\left(\frac{1}{m({B_{x}(r/2)})}
\fint\limits_{ T-\delta (r/2)^2}^{T+\delta (r/2)^2}\sum_{B_{x}(r/2)} mv_t^{2w}\drm t\right)^{\frac{1}{w}}.
\end{multline*}
We infer from Proposition~\ref{prop:doubling} and Proposition~\ref{prop:moser}
\begin{multline*}
\left(\frac{1}{m({B_{x}(r_1)})}\fint\limits_{ T-\delta r_1^2}^{T+\delta r_1^2}\sum_{B_{x}(r_1)} mv_t^{2q}\drm t\right)^{\frac{1}{q}}
\leq \left(\frac{1}{m({B_{x}(r_1)})}\fint\limits_{ T-\delta r_1^2}^{T+\delta r_1^2}\sum_{B_{x}(r_1)} mv_t^{2w}\drm t\right)^{\frac{1}{w}}
\\
\leq 
\left(\frac{m({B_{x}(r)})}{m(B_x(r_1))}
\right)^{\frac{1}{3q}}
\left(\frac{r^2}{r_1^2}\right)^{\frac{1}{w}}
 \frac{2^{ {110}n^2}\phi^{2n}\Phi_x(n,r)^{2n}(1+\delta r^2h(\omega))^{\frac{n}{2}+1}  }{\delta^{\frac{n}{2}+1} r^{2}m({B_{x}(r)})}
\int\limits_{T- \delta r^2}^{T+\delta r^2}\sum_{B_{x}(r)} mv_t^{2}\drm t,
\end{multline*}
where we used $1/w\leq 1/(3q)$ and the definition $\Phi_x(n,r)=\left[
{\frac{m(B_x(r))}{m(B_x(r_1))}}
\right]^{ 6\sqrt[n]{S/r}}$. We have
\[
\left(\frac{m({B_{x}(r)})}{m(B_x(r_1))}
\right)^{\frac{1}{3q}}\Phi_x(n,r)^{2n}
=
\left(\frac{m({B_{x}(r)})}{m(B_x(r_1))}\right)^{19n\sqrt[n]{S/r}}\leq 
\left(\frac{m({B_{x}(r)})}{m(B_x(r_1))}
\right)^{\frac{1}{q}}.
\]
Further, since $w(r)={\left(\frac{r}{2^3S}\right)^{\frac{1}{n}}}\geq \left(\frac{2^8S}{2^3S}\right)^{\frac{1}{n}}\geq 1$ we obtain
\[
\left(\frac{r^2}{ r_1^2}\right)^{\frac{1}{w(r)}}=\left(\frac{r}{ r_1}\right)^{\frac{w(r_1)}{w(r)}\frac{2}{w(r_1)}}=h(r/r_1)^{\frac{1}{w(r_1)}}\leq 2^{\frac{n^2}{w(r_1)}}\leq 2^{n^2},
\]
where we used $h(x)=x^{2/\sqrt[n]{x}}\leq 2^{n^2}$. 
This yields the desired $\ell^2$-$\ell^{q}$-estimate for non-negative $\omega$-solutions. 

Next, we apply the above estimate to the $\omega$-solution $v_t=P_t^\omega f$ for $0\leq f\in \ell^2(X,m)$ with $\|f\|_2\leq 1$. This yields 
 \begin{multline*}
\left(\frac{1}{m({B_{x}(r_1)})}\fint\limits_{ T-\delta r_1^2}^{T+\delta r_1^2}\|{B_{x}(r_1)} P_t^\omega f\|_{2q}^{2q}\drm t\right)^{\frac{1}{q}}
\\
\leq \left(\frac{m({B_{x}(r)})}{m(B_x(r_1))}
\right)^{\frac{1}{q}}\frac{C_{n,\phi}(1+\delta r^2h(\omega))^{\frac{n}{2}+1}  }{\delta^{\frac{n}{2}} m({B_{x}(r)})}
\fint\limits_{T- \delta r_2^2}^{T+\delta r_2^2}\|{B_{x}(r)} P_t^\omega f\|_{2}^{2}\drm t.
\end{multline*}
Clearly, $\|{B_{x}(r)} P_t^\omega f\|_{2}\leq \|{B_{x}(r)} P_t^\omega f\|_{2,2}$.
For general $f\in\ell^2(X,m)$ with $\|f\|_2\leq1$, we split $f=f_+-f_-$, where $f=f\vee 0$ and $f_-=(-f)\vee 0$ and apply the estimate above to each summand separately, resulting in an extra factor two. Taking the supremum on the left-hand side and noting that the map $\tau\mapsto \|B_x(r_1)P_\tau^\omega f\|_{w(r)}$ is continuous leads to the result by an application of the dominated convergence theorem.
\eat{\begin{multline*}
\left(\frac{m(B_x(r_1))}{m({B_{x}(r)})}\frac{r_1^2}{r^2}\right)^{\frac{1}{w}}\left(\frac{1}{m({B_{x}(r_1)})}\fint\limits_{ T-\delta r_1^2}^{T+\delta r_1^2}\sum_{B_{x}(r_1)} mv_t^{2w}\drm t\right)^{\frac{1}{w}}
\\
\leq \left(\frac{m(B_x(r_1))}{m({B_{x}(r/2)})}\frac{2r_1^2}{(r/2)^2}\right)^{\frac{1}{w}}\left(\frac{1}{m({B_{x}(r_1)})}\fint\limits_{ T-\delta r_1^2}^{T+\delta r_1^2}\sum_{B_{x}(r_1)} mv_t^{2w}\drm t\right)^{\frac{1}{w}}
\\
\leq 
\left(\frac{1}{m({B_{x}(r/2)})}\fint\limits_{ T-\delta (r/2)^2}^{T+\delta (r/2)^2}\sum_{B_{x}(r/2)} mv_t^{2w}\drm t\right)^{\frac{1}{w}}
\\
\leq 
 \Phi_x(n,r)^{2n}\frac{C_{n,\phi}(1+\delta r^2h(\omega))^{\frac{n}{2}+1}  }{\delta^{\frac{n}{2}} m({B_{x}(r)})}
\fint\limits_{T- \delta r^2}^{T+\delta r^2}\sum_{B_{x}(r)} mv_t^{2}\drm t
\end{multline*}
}
\end{proof}

\eat{In order to track the constants we decided to present the slightly adapted proof in the appendix. 
\begin{proposition}[Moser in time and space, cf.~{\cite[Theorem~2.7]{KellerRose-22a}}]\label{prop:moser}
Let $ x\in X $, $T\in\RR$, 
$n>2$, $\delta\in(0,1]$, $r_2\geq 2r_1\geq 256 S$ and constants $\phi,\Phi\geq 1$.
Assume $S_{\phi}(n,r_1,r_2)$ in $ x$.
 For all non-negative $\Delta_\omega$-subsolutions $v\geq 0$ on $[T-r_2^2,T+r_2^2]\times B_{x}(r_2)$, we have
 \[
\left(\frac{1}{m({B_{x}(r_1)})}\fint\limits_{ T-\delta r_1^2}^{T+\delta r_1^2}\sum_{B_{x}(r_1)} mv_t^{2w}\drm t\right)^{\frac{1}{w}}\leq C_{n,C_S}\left(\frac{r_2}{r_1}\right)^{2n^2}
 \frac{(1+\delta r_2^2h(\omega))^{\frac{n}{2}+1}  }{\delta^{\frac{n}{2}}m({B_{x}(r_2)})}
\fint\limits_{T- \delta r_2^2}^{T+\delta r_2^2}\sum_{B_{x}(r_2)} mv_t^{2}\drm t,
\]
where $C_{n,C_S}:=2^{19n^2}C_S^{2n}A'(n)^{2n}$ 
and 
\[
{w=w(r_2)={\left(\frac{r_2}{2^6S}\right)^{\frac{1}{n}}}.}
\]
\end{proposition}

\begin{proof}
We follow the proof of \cite[Theorem~2.7]{KellerR-22a}. Set $K=\lfloor\ln((r_2-r_1)/8S)\rfloor$ and 
\[
\rho_k:=r_1+2^{-k}(r_2-r_1),\qquad k\geq 0,
\]
and note that, since $ r_1\leq r_2/2$, for all $k\in\{0,\ldots, K\}$, we have 
\[
\rho_{k+1}+4S\leq \rho_k,\qquad 
\rho_k\geq \rho_{k+1}\geq\rho_k/2,\]
\[
\rho_{k}-\rho_{k+1}-2S\geq \frac{\rho_k-\rho_{k+1}}{2}=\frac{r_2-r_1}{2^{k+2}}\geq \frac{r_2}{2^{k+3}},
\]
and
\[\rho_k^2-\rho_{k+1}^2=(\rho_k+\rho_{k+1})(\rho_k-\rho_{k+1})\geq \frac{r_2-r_1}{2^k}\frac{r_2-r_1}{2^{k+1}}=\frac{(r_2-r_1)^2}{2^{2k+2}}\geq \frac{r_2^2}{2^{2k+4}}.
\]
Therefore, we can perform all calculations until the bottom of p.~20 in \cite{KR24a} and obtain
\[
\left(\frac{1}{m(B(r_1))}\fint_{T-\delta r_1^2}^{T+\delta r_1^2}\sum_{B(r_1)}mv_t^{2\alpha^K}\right)^{1/\alpha^K}
\leq 
\frac{C_1}{m(B(r_2))}
\fint_{T-\delta r_2^2}^{T+\delta r_2^2}\sum_{B(r_2)}mv_t^{2}
\]
where 
\[
C_1=\left(\frac{\rho_K^2}{r_1^2}
\frac{m(B(\rho_K))}{m(B(r_1))}\right)^{1/\alpha^K}
\prod_{k=0}^{K-1}C_{0,k}^{1/\alpha^{k+1}}
\]
and 
\begin{multline*}
C_{0,k}=\frac{5}{2}C_S999^\alpha\alpha^{2k\alpha}\frac{(\rho_k+\rho_{k+1})^2}{4}\frac{m(B(\rho_k))}{m(B(\rho_{k+1}))}\left(\frac{m(B(\rho_k))}{m(B(\tfrac{\rho_k+\rho_{k+1}}{2}))}\right)^{\frac{2}{n}}
\\
\cdot \frac{(2\delta\rho_k^2)^\alpha}{\delta(\rho_k^2+\rho_{k+1}^2)}
\left(h(\omega)+\frac{8}{(\rho_k-\rho_{k+1}-2S)^2}+\frac{1}{\delta(\rho_k^2-\rho_{k+1}^2)}\right)^\alpha
\end{multline*}
We have 
\[
\left(\frac{\rho_K^2}{r_1^2}
\frac{m(B(\rho_K))}{m(B(r_1))}\right)^{1/\alpha^K}\leq \left(\frac{r_2^2}{r_1^2}
\frac{m(B(r_2))}{m(B(r_1))}\right)^{1/\alpha^K}\leq \left(\frac{r_2^2}{r_1^2}
\Phi\right)^{1/\alpha^K}.
\]
Further, by the estimates above and $\alpha\leq 2$, we have 
\[
\frac{(\rho_k+\rho_{k+1})^2}{4}\leq \rho_k^2\le r_2^2,\qquad \frac{(2\delta\rho_k^2)^\alpha}{\delta(\rho_k^2+\rho_{k+1}^2)}\leq 2^\alpha (\delta \rho_k^2)^{\alpha-1}\leq 2^2 (\delta r_2^2)^{\alpha-1},
\]
and by the doubling property, $\rho_k/2\geq r_1$, and domain monotonicity of $m$
\[
\frac{m(B(\rho_k))}{m(B(\rho_{k+1}))}
\leq \frac{m(B(r_2))}{m(B(r_1))}\leq \Phi,\qquad 
\frac{m(B(\rho_k))}{m(B(\tfrac{\rho_k+\rho_{k+1}}{2}))}\leq \frac{m(B(\rho_k))}{m(B(\tfrac{\rho_k}{2}))}\leq \Phi.
\]
Hence, using $\alpha=1+2/n\leq 2$ and $5/2 \cdot 999^\alpha\cdot \alpha^{2k\alpha}\leq 2^2\cdot 2^{10\alpha}\cdot 2^{2k\alpha}\leq 2^2\cdot 2^{20}\cdot 2^{4k}=2^{22+4k}$, we obtain, using $ r_1\leq r_2/2$, 
\begin{multline*}
C_{0,k}\leq 2^{22+4k}C_Sr_2^2\Phi^\alpha
2^2(\delta r_2^2)^{\alpha-1}
\left(h(\omega)+\frac{82^{2k+6}}{r_2^2}+\frac{2^{2k+4}}{\delta r_2^2}\right)^\alpha
\\
\leq 
2^{22+4k}C_Sr_2^2\Phi^\alpha
2^2(\delta r_2^2)^{\alpha-1}
\left(h(\omega)+\frac{2^{2k+9}}{\delta r_2^2}+\frac{2^{2k+9}}{\delta r_2^2}\right)^\alpha
\\
=2^{24+9\alpha+(4+2\alpha)k}C_S\Phi^\alpha(1+\delta r_2^2h(\omega))^\alpha.
\end{multline*}
Hence, since  $\sum_{k=0}^{K-1}\frac{1}{\alpha^{k+1}}=\frac{n}{2}(1-\alpha^{-K})\leq \frac{n}{2}$ and $\sum_{k=0}^{K-1}\frac{k}{\alpha^k}\leq \frac{\alpha^2}{(\alpha-1)^2}\leq n^2$, and $1<n/2$,
\begin{multline*}
\prod_{k=0}^{K-1}C_{0,k}^{1/\alpha^{k+1}}\leq 
\prod_{k=0}^{K-1}\left(2^{24+9\alpha+(4+2\alpha)k}C_S\Phi^\alpha(1+\delta r_2^2h(\omega))^\alpha\right)^{1/\alpha^{k+1}}
\\
=\left(2^{24+9\alpha}C_S\Phi^\alpha(1+\delta r_2^2h(\omega))^\alpha\right)^{\sum_{k=0}^{K-1}\frac{1}{\alpha^{k+1}}}2^{(4+2\alpha)\sum_{k=0}^{K-1}\frac{k}{\alpha^{k+1}}}
\\
\leq 
\left(2^{24+9\alpha}C_S\Phi^\alpha(1+\delta r_2^2h(\omega))^\alpha\right)^{\frac{n}{2}}2^{(4+2\alpha)n^2}
\\
=
2^{6n^2+\frac{73n}{4}+\frac92}C_S^\frac{n}{2}\Phi^{\frac{n}{2}+1}(1+\delta r_2^2h(\omega))^{\frac{n}{2}+1}
\leq 
2^{6n^2+\frac{73n^2}{8}+\frac98n^2}C_S^n\Phi^{n}(1+\delta r_2^2h(\omega))^{\frac{n}{2}+1}
\\
=2^{\frac{130n^2}{8}}C_S^n\Phi^{n}(1+\delta r_2^2h(\omega))^{\frac{n}{2}+1}
\leq 2^{17n^2}C_S^n\Phi^{n}(1+\delta r_2^2h(\omega))^{\frac{n}{2}+1}.
\end{multline*}
Hence
\begin{multline*}
C_1=\left(\frac{\rho_K^2}{r_1^2}
\frac{m(B(\rho_K))}{m(B(r_1))}\right)^{1/\alpha^K}
\prod_{k=0}^{K-1}C_{0,k}^{1/\alpha^{k+1}}
\leq 
\left(\frac{r_2^2}{r_1^2}
\Phi\right)^{1/\alpha^K}2^{17n^2}C_S^n\Phi^{n}(1+\delta r_2^2h(\omega))^{\frac{n}{2}+1}
\\
=
\left(\frac{r_2^2}{r_1^2}\right)^{1/\alpha^K}2^{17n^2}C_S^n\Phi^{n+1/\alpha^K}(1+\delta r_2^2h(\omega))^{\frac{n}{2}+1}
\leq \left(\frac{r_2^2}{r_1^2}\right)^{1/\alpha^K}2^{17n^2}(C_S\Phi)^{2n}(1+\delta r_2^2h(\omega))^{\frac{n}{2}+1}
\end{multline*}
Next, since $K\geq \ln(r_2-r_1)/8S -1=\ln (r_2-r_1)/8S\euler\geq \ln r_2/(2^6S)$ and $1\geq 2^6S/r_2$ and $\ln \alpha\geq 2/(n+2)\geq 1/n $ since $2<n$, we have 
\[
\frac{1}{\alpha^K}\leq \frac{1}{\alpha^{\ln r_2/(2^6S)}}=\left(\frac{2^6S}{r_2}\right)^{\ln\alpha}\leq \left(\frac{2^6S}{r_2}\right)^{\frac{1}{n}}
\]
Since $r_1\geq 2^6S$ we have 
\[
\left(\frac{r_2^2}{r_1^2}\right)^{1/\alpha^K}\leq h(r_2/2^6S)
\] 
where $h(x)=x^{2/x^{1/n}}\leq h(\euler^n)=\euler^{2n/\euler}\leq 2^{2n}\leq 2^{n^2}$.
Theorem~\ref{thm:adapteddoubling} yields the claim with $\Phi=A'(n)(r_2/r_1)^n$.
\end{proof}

}

\eat{
\begin{corollary}[Moser in time and space 2]\label{thm:KR22-subsolution2}
Let $ x\in X $, $T\in\RR$, 
$n>2$, $\delta\in(0,1]$, $r_2\geq 2r_1\geq 256S$. Assume $S(n,r_1,r_2)$ in $x$.
 Then, for all $\omega\in \ell^\infty(X)$, we have 
 \[
\left(\frac{1}{m({B_{x}(r_1)})}\fint\limits_{ T-\delta r_1^2}^{T+\delta r_1^2}\|{B_{x}(r_1)} P_t^\omega\|_{2,2w(r_2)}^{2w(r_2)}\drm t\right)^{\frac{1}{w(r_2)}}
\leq C
 \frac{(1+\delta r_2^2h(\omega))^{\frac{n}{2}+1}  }{\delta^{\frac{n}{2}} m({B_{x}(r_2)})}
\fint\limits_{T- \delta r_2^2}^{T+\delta r_2^2}\|{B_{x}(r_2)} P_t^\omega\|_{2,2}^{2}\drm t.
\]
\end{corollary}
}

\subsection{Integrated Davies trick and averaged heat kernel}\label{sec:davies}

In this section we  derive averaged Gaussian upper heat kernel bounds from $\ell^2-\ell^p$-bounds for positive solutions of sandwiched heat equations via an integrated version of Davies' trick. Recall that in Section~\ref{sec:FKVM} we defined 
for 
$\omega\in \ell^\infty(X)$ the sandwiched semigroup $(P_t^{\omega})_{t\geq 0}$ in $\ell^2(X,m)$ by
\[
P_t^{\omega}:=\euler^{\omega}P_t\euler^{-\omega}.
\]

We have the following general bound for the operator norms of the sandwiched semigroup. 
\begin{theorem}[Integrated Davies trick]\label{lem:integratedDavies} For any $q\in[1,\infty]$, denote by $q'$ its H\"older conjugate and let $0\leq V,W\in\ell^\infty(X),$ $I\subset \RR$. Then, for any $q$-norm with respect to any measure on $X$, we have 
\begin{multline*}
\left(
\fint_I \|VP_{2\tau}W\|_{q',q}^\frac{q}{2}\drm \tau 
\right)^\frac{2}{q}
=
\left(
\fint_I \|WP_\tau V\|_{q',q}^\frac{q}{2}\drm \tau 
\right)^{\frac{2}{q}}
\\
\leq 
\inf_{\omega\in\ell^\infty(X)}
\sup_{\supp V}\euler^{-\omega}
\sup_{\supp W}\euler^{\omega}
\left(
\fint_I
\|V P_\tau^\omega\|_{2,q}^q
\drm \tau
\right)^\frac{1}{q}
\left(
\fint_I
\|W P_\tau^{-\omega}\|_{2,q}^q
\drm \tau
\right)^\frac{1}{q}.
\end{multline*}
\end{theorem}
\begin{proof}
The first identity follows directly from duality, and we are left with the proof of the estimate. Let $\omega\in\ell^\infty(X)$. Since $P_{2t}$ is positivity preserving for any $t>0$, we obtain
\[
\|VP_{2\tau}W\|_{q',q}=\|V\euler^{-\omega}P_{2\tau}\euler^{\omega}W\|_{q',q}\leq \sup_{\supp V}\euler^{-\omega}
\sup_{\supp W}\euler^{\omega}\|VP_{2\tau}^\omega W\|_{q',q}.
\]
Further, the semigroup property of $(P_t^\omega)_{t\geq 0}$ implies 
\[
\|VP_{2\tau}^\omega W\|_{q',q}=\|VP_\tau^\omega P_\tau^\omega W\|_{q',q}\leq \|VP_\tau^\omega\|_{2,q}\|P_\tau^\omega W\|_{q',2}.
\]
As the adjoint of $P_\tau W$ is $WP_\tau^{-\omega}$, we have 
\[
\|P_\tau^\omega W\|_{q',2}=\|WP_\tau^{-\omega}\|_{2,q}.
\]
Therefore, by the Cauchy-Schwarz inequality,
\begin{multline*}
\left(
\fint_I \|VP_{2\tau}W\|_{q',q}^\frac{q}{2}\drm \tau 
\right)^\frac{2}{q}
\leq 
\sup_{\supp V}\euler^{-\omega}
\sup_{\supp W}\euler^{\omega}
\left(
\fint_I \|VP_{2\tau}^\omega W\|_{q',q}^\frac{q}{2}\drm \tau 
\right)^\frac{2}{q}
\\
\leq 
\sup_{\supp V}\euler^{-\omega}
\sup_{\supp W}\euler^{\omega}
\left(
\fint_I \|VP_\tau^\omega\|_{2,q}^\frac{q}{2}\|WP_\tau^{-\omega}\|_{2,q}^\frac{q}{2}\drm \tau 
\right)^\frac{2}{q}
\\
\leq 
\sup_{\supp V}\euler^{-\omega}
\sup_{\supp W}\euler^{\omega}
\left(
\fint_I \|VP_\tau^\omega\|_{2,q}^q\right)^\frac{1}{q}
\left(\fint_I\|WP_\tau^{-\omega}\|_{2,q}^q\drm \tau 
\right)^\frac{1}{q}.
\end{multline*}
Since $\omega\in\ell^\infty(X)$ was arbitrary, the claim follows.
\end{proof}

\eat{
For subsets $\emptyset\neq A,B\subset X$ and $p,q\in[1,\infty]$, and a bounded  operator $P\colon \ell^{q'}\left(A,\frac{m}{m(A)}\right)\to \ell^q\left(B,\frac{m}{m(B)}\right)$, we denote
\[
\tvert{P}_{A,B,p,q}:=\left\lVert P\colon \ell^{p}\left(A,\frac{m}{m(A)}\right)\to \ell^q\left(B,\frac{m}{m(B)}\right)\right\rVert.
\]

\begin{corollary}[Integrated Davies trick - normalized]\label{cor:integratedDaviesNorm} For any $q\in[1,\infty]$, denote by $q'$ its H\"older conjugate and let $A,B\subset X$, and $I\subset \RR$. Then, for any $q$-norm with respect to any measure on $X$, we have 
\begin{multline*}
\left(
\fint_I \tvert{P_{2\tau}}_{A,B,q',q}\drm \tau 
\right)^\frac{2}{q}
\\
\leq 
\inf_{\omega\in\ell^\infty(X)}
\sup_{A}\euler^{-\omega}
\sup_{B}\euler^{\omega}
\left(
\fint_I
\tvert{P_\tau^\omega}_{A,A,2,q}^q
\drm \tau
\right)^\frac{1}{q}
\left(
\fint_I
\tvert{P_\tau^{-\omega}}_{B,B,2,q}^q
\drm \tau
\right)^\frac{1}{q}.
\end{multline*}
\end{corollary}

}


The upper bound in Theorem~\ref{lem:integratedDavies} depends on the best choice of a function in a minimizaton problem. In the pointwise heat kernel bound case, Davies,  \cite{Davies-93,Davies-87}, chose a certain family of Lipschitz functions and minimized with respect to the Lipschitz constant. We will use this argument as well.\\
We need an integrated maximum principle for graphs with intrinsic metric proven in \cite{BauerHuaYau-17}. We say that $\omega \in \cC(X)$ is $\kappa$-Lipschitz if $\omega$ is Lipschitz continuous with respect to the intrinsic metric $\rho$ with Lipschitz constant $\kappa$. Recall that $ S $ denotes the jump size of the intrinsic metric and
\[
\Lambda=\inf \spec(\Delta).
\]
We have the following.
\begin{proposition}[{\cite[Lemma~3.3]{BauerHuaYau-17}}]\label{lem:cauchy}Let $\kappa\geq 0$ and $f\in \ell^2(X,m)$. For any $\kappa$-Lipschitz function $\omega\in \cC(X)$, the function 
\begin{equation*}
\Phi:[0,\infty)\to[0,\infty),\quad \ t\mapsto \exp\left(2\Lambda t -\frac{2}{S^2}\left(\cosh\left(\kappa S\right)-1\right)t\right)\ \Vert \euler^{\omega}P_tf\Vert_2^2
\end{equation*}
is non-increasing.
\end{proposition}
Recall from Section~\ref{sec:FKVM} the quantity
\[
h(\omega)=\sup_{x\in X}\sum_{x\in X}\frac{b(x,y)}{m(x)}\vert (\euler^{\omega(x)}-\euler^{\omega(y)})(\euler^{-\omega(x)}-\euler^{-\omega(y)})\vert
\]
for any $\omega\in\ell^\infty(X)$.
The proof of the following proposition is contained in the proof of \cite[Theorem~5.3]{KellerRose-22a}.
\begin{proposition}[\cite{KellerRose-22a}]\label{prop:BHY}For $b\geq a\geq 0$, $\kappa>0$ and $\omega\in\ell^\infty(X)$ $\kappa$-Lipschitz, we have 
\begin{align*}
\fint_{a}^{b}\Vert P_\tau^{\omega}\Vert_{2,2}^2 \drm \tau
\leq \euler^{-2\Lambda a+\gamma(\kappa S)b},
\end{align*}
and 
\[
h(\omega)\leq \gamma(\kappa S)
:=\frac{2}{S^2}\left(\cosh\left(\kappa S\right)-1\right).
\]
\end{proposition}

The next theorem provides the technical step in finding off-diagonal averaged heat kernel bounds by finding an appropriate candidate for the optimization problem in Theorem~\ref{lem:integratedDavies}.

\begin{theorem}\label{thm:abstractbound}
Let $x,y\in X$, $p\in[1,\infty]$, $r_2\geq r_1\geq0$, $t>0$, and $\delta\in(0,1)$. Assume that for $z\in\{x,y\}$ and all $\omega\in\ell^\infty(X)$ there are $\phi_z(h)=\phi_z(p,r_1,r_2,\delta,t,h)>0$, monotone increasing in $h$, such that 
\[
\left(\fint_{t/2-\delta r_1^2}^{t/2+\delta r_1^2}\|{B_z(r_1)}P_\tau^\omega\|_{2,2p}^{2p}\drm \tau\right)^\frac{1}{p}
\leq \phi_z(h(\omega))^2\fint_{t/2-\delta  r_2^2}^{t/2+\delta r_2^2}\| P_\tau^\omega\|_{2,2}^2\drm \tau.
\] 
Then, for $t\geq 2r_1^2$,  we have
\[
\left(
\fint_{t-2\delta r_1^2}^{t+2\delta r_1^2} \|{B_x(r_1)}P_{\tau}{B_y(r_1)}\|_{(2p)',2p}^p\drm \tau
\right)^\frac{1}{p}
\leq 
\phi_x(\sigma)\phi_y(\sigma)\euler^{\sigma(\delta r_2^2-t/2)-\Lambda(t-2\delta r_1^2)-tS^{-2}\zeta(\rho(B_x(r_1),B_y(r_1))S/t)},
\]
where 
\[
\sigma=\sigma(\rho(B_x(r_1),B_y(r_1))S/t),\qquad \sigma(x)=\sqrt{1+x^2}-1.
\]
\end{theorem}

\begin{proof}
Apply Lemma~\ref{lem:integratedDavies} with $q=2p$, $V= {B_x(r_1)}$, $W= {B_y(r_1)}$, and $I=[t/2-\delta r_1^2,t/2+\delta r_1^2]$ to obtain
\begin{multline*}
\left(
\fint_{t/2-\delta r_1^2}^{t/2+\delta r_1^2} \| {B_x(r_1)}P_{2\tau} {B_y(r_1)}\|_{(2p)',2p}^p\drm \tau
\right)^\frac{1}{p}
\\
\leq 
\inf_{\omega\in\ell^\infty(X)}
\sup_{B_y(r_1)}\euler^{\omega}\sup_{B_x(r_1)}\euler^{-\omega}
\left(
\fint_I
\| {B_x(r_1)} P_\tau^\omega\|_{2,2p}^{2p}
\drm \tau
\right)^\frac{1}{2p}
\left(
\fint_I
\| {B_y(r_1)} P_\tau^{-\omega}\|_{2,2p}^{2p}
\drm \tau
\right)^\frac{1}{2p}.
\end{multline*}
For all $\kappa>0$, $\omega\in\ell^\infty(X)$ $\kappa$-Lipschitz, the assumption together with the integrated maximum principle from Proposition~\ref{prop:BHY} yields for $z\in\{x,y\}$
\[
\left(\fint_{t/2-\delta r_1^2}^{t/2+\delta r_1^2}\| {B_z(r_1)}P_\tau^\omega\|_{2,2p}^{2p}\drm \tau\right)^\frac{1}{2p}
\leq \phi_z(\gamma(\kappa S))\euler^{-\Lambda(t/2-\delta r_2^2)+\gamma(\kappa S)(t/4+\delta r_2^2/2)}.
\]
Note that if $\omega$ is $\kappa$-Lipschitz and above inequality holds, then $-\omega$ is $\kappa$-Lipschitz, and above estimate holds for $-\omega$ as well.

The same bound holds if we replace $\omega$ by $\hat \omega$. 

Next, for $\kappa>0$ we let
\[
\omega=-\kappa\min\{\rho(B_x(r_1),\cdot),\rho(B_x(r_1),B_y(r_1))\}.
\]
Then $\omega\in \ell^\infty$ is $\kappa$-Lipschitz and
\[
\sup_{B_y(r_1)}\euler^{\omega}\sup_{B_x(r_1)}\euler^{-\omega}
= \euler^{-\kappa\rho(B_x(r_1),B_y(r_1))}.
\]
Plugging in these infos leads to 
\begin{multline*}
\left(
\fint_{t/2-\delta r_1^2}^{t/2+\delta r_1^2} \| {B_x(r_1)}P_{2\tau} {B_y(r_1)}\|_{(2p)',2p}^p\drm \tau
\right)^\frac{1}{p}
\\
\leq 
\inf_{\substack{\omega\in\ell^\infty(X),\\ \omega \ \kappa-\text{Lipschitz}}}
\phi_x(\gamma(\kappa S))\phi_y(\gamma(\kappa S))\euler^{\gamma(\kappa S)(\delta r_2^2-t/2)-\Lambda(t-2\delta r_2^2)-\kappa\rho(B_x(r_1),B_y(r_1))+t\gamma(\kappa S)}
.
\end{multline*}
The minimizer of the function 
$
[0,\infty)\ni \kappa \mapsto -\kappa\rho(B_x(r_1),B_y(r_1))+t\gamma(\kappa S)
$
is 
\[
\kappa_0 =\frac{1}{S^2}\arsinh\left(\frac{\rho(B_x(r_1),B_y(r_1))S}{t}\right).
\]
We obtain
\[
\gamma(\kappa_0 S)=\sqrt{1+\left(\frac{\rho(B_x(r_1),B_y(r_1))S}{t}\right)^2}-1=\sigma(\rho(B_x(r_1),B_y(r_1))S/t)=\sigma.
\]
Hence
\begin{multline*}
\left(
\fint_{t/2-\delta r_1^2}^{t/2+\delta r_1^2} \| {B_x(r_1)}P_{2\tau} {B_y(r_1)}\|_{(2p)',2p}^p\drm \tau
\right)^\frac{1}{p}
\\
\leq
\phi_x(\sigma)\phi_y(\sigma)\euler^{\sigma(\delta r_2^2-t/2)-\Lambda(t-2\delta r_1^2)-tS^{-2}\zeta(\rho(B_x(r_1),B_y(r_1))S/t)}.
\end{multline*}
Finally, subsitute $\tau\to 2\tau$ in the integral on the left-hand side to conclude.
\end{proof}

\subsection{Faber-Krahn yields averaged heat kernel bounds}\label{sec:proofmain}

In the following, we choose particular constants in Theorem~\ref{thm:abstractbound} in order to obtain the first main result Theorem~\ref{thm:main_FK_hk} and Theorem~\ref{thm:FK_hk_global}, i.e., showing that Faber-Krahn on large scales yields averaged Gaussian upper heat kernel bounds.

We start with a lemma providing averaged Gaussian upper bounds for a fixed time provided a certain $\ell^2-\ell^p$-estimate holds.

\begin{lemma}\label{lem:abstractbound2}
Let $x,y\in X$, $p\in[1,\infty]$, $C,t>0$, $r_2\geq r_1\geq 0$. Assume that for $z\in\{x,y\}$, all $\delta\in(0,1)$ and all $\omega\in\ell^\infty(X)$ we have
\[
\left(\fint_{t/2-\delta r_1^2}^{t/2+\delta r_1^2}\| {B_z(r_1)}P_\tau^\omega\|_{2,2p}^{2p}\drm \tau\right)^\frac{1}{p}
\leq C(z)\delta^{-\frac{n}{2}}(1+\delta r_2^2 h(\omega))^{\frac{n}{2}+1}\fint_{t/2-\delta r_2^2}^{t/2+\delta r_2^2}\| P_\tau^\omega\|_{2,2}^2\drm \tau.
\]
Then we have 
\begin{multline*}
\left(
\fint_{t-\frac{r_1^2}{1\vee r_2^2\sigma}}^{t+\frac{r_1^2}{1\vee r_2^2\sigma}} \| {B_x(r_1)}P_{t} {B_y(r_1)}\|_{(2p)',2p}^p\drm t 
\right)^\frac{1}{p}
\\
\leq 
2^{n+5}\sqrt{C(x)C(y)}
(1\vee r_2^2\sigma)^{\frac{n}{2}}
 \euler^{-\Lambda(t-r_2^2)-tS^{-2}\zeta(\rho(B_x(r_1),B_y(r_1))S/t)}
\end{multline*}
where $\sigma=\sigma(\rho(B_x(r_1),B_y(r_1))S/t)$ is given in Theorem~\ref{thm:abstractbound}.
\end{lemma}
\begin{proof}Abbreviate $$\phi_z(h(\omega))^2=C(z,r_1,r_2)(1+h(\omega))^{\frac{n}{2}+1}\delta^{-\frac{n}{2}}.$$
Theorem~\ref{thm:abstractbound} implies
\begin{multline*}
\left(
\fint_{t-2\delta r_1^2}^{t+2\delta r_1^2} \| {B_x(r_1)}P_{t} {B_y(r_1)}\|_{(2p)',2p}^p\drm t 
\right)^\frac{1}{p}
\leq 
\phi_x(\sigma)\phi_y(\sigma)\euler^{-\Lambda(t-2\delta r^2)-tS^{-2}\zeta(\rho(B_x(r_1),B_y(r_1)S/t)}
\\
=
\sqrt{C(x)C(y)}
(1+\delta r_2^2\sigma)^{\frac{n}{2}+1}\delta^{-\frac{n}{2}}
 \euler^{\sigma(\delta r_2^2-t/2)-\Lambda(t-2\delta r_2^2)-tS^{-2}\zeta(\rho(B_x(r_1),B_y(r_1))S/t)},
\end{multline*}
\eat{where we recall
$
\sigma=\sigma(\rho(B_x(r_1),B_y(r_1))S/t),
$
$\sigma(x)=\sqrt{1+x^2}-1$.
}
Next, choose
\[
\delta=\frac{1}{2}\wedge \frac{1}{r_2^2\sigma}
\]
and use the bound $\sigma(\delta r_2^2-t/2)\leq (1/\sigma-t/2)\sigma\leq 1$ as well as $\delta\leq 1/2$ to obtain
\[
\sigma (\delta r_2^2-t/2)-\Lambda(t-2\delta r_2^2)
\leq 1-\Lambda(t-r_2^2).
\]
Finally, use the estimates $\delta r_2^2\sigma\leq 1$, $\euler\leq 2^2$,
\[
\frac{1}{4}\fint_{t-\frac{r_1^2}{1\vee r_2^2\sigma}}^{t+\frac{r_1^2}{1\vee r_2^2\sigma}}f\drm \tau\leq \fint_{t-2\delta r_1^2}^{t+2\delta r_1^2}f\drm \tau 
\]
for any non-negative measurable $f$, and $4^{1/p}\leq 2^2$ to conclude.
\end{proof}

We are now in the position to prove the averaged heat kernel bound which is our first main result.
\begin{theorem}\label{thm:main_FK_hk}
Let $ x,y\in X $, $t\in\RR$, 
$n>2$, $\delta\in(0,1]$, $r_2\geq 2r_1\geq 256S$.Assume that, for all $B(r)\subset  X$, $r\in [r_1,r_2]$, and all $U\subset B(r)$, we have
\begin{equation*}
\lambda(U)\geq \frac{C}{r^2}\left(\frac{m(B(r))}{m(U)}\right)^\frac{2}{n}.
\end{equation*} Then, for all $t\geq 2r_1^2$, we have 
\begin{multline*}
\mathcal{P}_{r_1}(t,x,y)
 \\
 \leq 
\frac{2^{5n^2}C_{n,C_S}}{ \sqrt[2q(\tau)]{m({B_{x}(r_1)})m({B_{y}(r_1)})}}
\frac{(1\vee \tau\sigma)^{\frac{n}{2}}}{\sqrt[\frac{2q(\tau)}{q( \tau)-1}]{m(B_x(\sqrt \tau))m(B_y(\sqrt \tau))}}
 \euler^{-\Lambda(t-\tau)-tS^{-2}\zeta(\rho(B_x(r_1),B_y(r_1))S/t)}
\end{multline*}
where $\sigma=\sigma(\rho(B_x(r_1),B_y(r_1))S/t)$, $\tau= t\wedge r_2^2$, and 
\[
q(\tau)=\frac{1}{21n}\sqrt[2n]{\frac{\tau}{S^2}}.
\]
\end{theorem}
\begin{proof}By Proposition~\ref{prop:FKS} we have the Sobolev inequality in the same range of parameters with a slightly different constant and the doubling property by Proposition~\ref{prop:doubling}.
Apply Theorem~\ref{thm:subsolution} with $r_1=r_1$ and $r_2'=\sqrt \tau$ to obtain, for $z\in\{x,y\}$ and all $t\geq 2r_1^2$, $\delta\in(0,1)$ and $\omega\in \ell^\infty(X)$, 
 \begin{multline*}
\left(\frac{1}{m({B_{z}(r_1)})}\fint\limits_{ t/2-\delta r_1^2}^{t/2+\delta r_1^2}\|{B_{z}(r_1)} P_s^\omega\|_{2,2q( \tau)}^{2q(\tau)}\drm s\right)^{\frac{1}{q( \tau)}}
\\
\leq C(z,\sqrt \tau)
 (1+\delta \tau h(\omega))^{\frac{n}{2}+1} \delta^{-\frac{n}{2}} 
\fint\limits_{t- \delta \tau}^{t+\delta \tau}\|{B_{z}(\sqrt \tau)} P_t^\omega\|_{2,2}^{2}\drm t,
\end{multline*}
where
\[
C(z,\sqrt \tau)= C_{n,C_S}\left(\frac{m({B_{z}(\sqrt \tau)})}{m({B_{z}(r_1)})}\right)^{\frac{1}{q(\tau)}}\frac{1}{m({B_{z}(\sqrt \tau)})}
\]
Hence, Lemma~\ref{lem:abstractbound2} implies 
\begin{multline*}
\mathcal{P}_{r_1}(t,x,y)=\left(
\fint_{t-\frac{r_1^2}{1\vee \tau\sigma}}^{t+\frac{r_1^2}{1\vee \tau\sigma}} \| {B_x(r_1)}P_{\tau} {B_y(r_1)}\|_{(2q(\tau))',2q(\tau)}^{q(\tau)}\drm \tau
\right)^\frac{1}{q(\tau)}
\\
\leq 
2^{n+5}\sqrt{C(x,\sqrt \tau)C(y,\sqrt \tau)}
(1\vee \tau\sigma)^{\frac{n}{2}}
 \euler^{-\Lambda(t-\tau)-t\zeta(\rho(B_x(r_1),B_y(r_1))S/t)}
\\
 \leq 
2^{n+5}C_{n,C_S}\left(\frac{m({B_{x}(\sqrt \tau)})}{m({B_{x}(r_1)})}\right)^{\frac{1}{2q(\tau)}}
\left(\frac{m({B_{y}(\sqrt \tau)})}{m({B_{y}(r_1)})}\right)^{\frac{1}{2q(\tau)}}
\\
\cdot
\frac{(1\vee \tau\sigma)^{\frac{n}{2}}}{\sqrt{m(B_x(\sqrt \tau))m(B_y(\sqrt \tau))}}
 \euler^{-\Lambda(t-\tau)-tS^{-2}\zeta(\rho(B_x(r_1),B_y(r_1))S/t)}
\end{multline*}
The estimate $\sqrt\tau=\sqrt t\wedge r_2\leq \sqrt t$ in the factor involving $\sigma$ leads to  the claim. 
\end{proof}

\eat{
\begin{lemma}
Let $ x\in X $, $T\in\RR$, 
$n>2$, $\delta\in(0,1]$, $r_2\geq 2r_1\geq 256 S$ and constants $\phi,\Phi\geq 1$.
Assume $S_{\phi}(n,r_1,r_2)$ in $ x$.
 For all non-negative $\Delta_\omega$-subsolutions $v\geq 0$ on $[T-r_2^2,T+r_2^2]\times B_{x}(r_2)$, we have
 \begin{multline*}
\left(\frac{1}{2  \delta r_1^2m({B_{x}(r_1)})}\int\limits_{ T-\delta r_1^2}^{T+\delta r_1^2}\sum_{B_{x}(r_1)} mv_t^{2w(r_2)}\drm t\right)^{\frac{1}{w(r_2)}}
\\
\leq C\left(\frac{m({B_{x}(r_2)})}{m({B_{x}(r_1)})}\right)^{\frac{1}{w(r_2)}}
 \frac{(1+\delta r_2^2h(\omega))^{\frac{n}{2}+1}  }{\delta^{\frac{n}{2}+1} r_2^{2}m({B_{x}(r_2)})}
\int\limits_{T- \delta r_2^2}^{T+\delta r_2^2}\sum_{B_{x}(r_2)} mv_t^{2}\drm t
\end{multline*}
where  $2^{3n^2}C_{n,C_S}$ and $
w(r)={\left(\frac{r}{2^6S}\right)^{\frac{1}{n}}}$.
\end{lemma}
\begin{proof}
Apply Proposition~\ref{prop:moser} with $r_1'=r_2/2$ and $r_2'=r_2$ to obtain 
 \begin{multline*}
 \left(\frac{1}{2  \delta r_1^2m({B_{x}(r_1)})}\int\limits_{ T-\delta r_1^2}^{T+\delta r_1^2}\sum_{B_{x}(r_1)} mv_t^{2w(r_2)}\drm t\right)^{\frac{1}{w(r_2)}}
 \\
 \leq \left(\frac{(r_2/2)^2m({B_{x}(r_2/2)})}{ r_1^2m({B_{x}(r_1)})}\right)^{\frac{1}{w(r_2)}}
\left(\frac{1}{2  \delta (r_2/2)^2m({B_{x}(r_2/2)})}\int\limits_{ T-\delta (r_2/2)^2}^{T+\delta (r_2/2)^2}\sum_{B_{x}(r_2/2)} mv_t^{2w(r_2)}\drm t\right)^{\frac{1}{w(r_2)}}
\\
\leq 2^{2n^2}C_{n,C_S}\left(\frac{r_2^2m({B_{x}(r_2)})}{ r_1^2m({B_{x}(r_1)})}\right)^{\frac{1}{w(r_2)}}
 \frac{(1+\delta r_2^2h(\omega))^{\frac{n}{2}+1}  }{\delta^{\frac{n}{2}+1} r_2^{2}m({B_{x}(r_2)})}
\int\limits_{T- \delta r_2^2}^{T+\delta r_2^2}\sum_{B_{x}(r_2)} mv_t^{2}\drm t.
\end{multline*}
Recall $w(r_2)={\left(\frac{r_2}{2^6S}\right)^{\frac{1}{n}}}\geq \left(\frac{2^8S}{2^6S}\right)^{\frac{1}{n}}\geq 1$ such that 
\[
\left(\frac{r_2^2}{ r_1^2}\right)^{\frac{1}{w(r_2)}}=\left(\frac{r_2}{ r_1}\right)^{\frac{w(r_1)}{w(r_2)}\frac{2}{w(r_1)}}=h(r_2/r_1)^{\frac{1}{w(r_1)}}\leq 2^{\frac{n^2}{w(r_1)}}\leq 2^{n^2},
\]
where we used $h(x)=x^{2/\sqrt[n]{x}}\leq 2^{n^2}$. This yields the claim. 
\end{proof}

}

\begin{proof}[Proof of
Theorem~\ref{thm:FK_hk_global}]
This follows from Theorem~\ref{thm:main_FK_hk} and Proposition~\ref{prop:doubling} by setting $r_2=\infty$. 
\end{proof}
\eat{\begin{corollary}
Let  
$n>2$,  $r_0\geq 128S$. Assume $S(n,r_0,\infty)$ in every vertex of $X$. Then there is a constant $C>0$ such that, for all $r_1\geq r_0$ and $t\geq 2r_1^2$, $x,y\in X$, we have 
\begin{multline*}
\mathcal{P}_{r_1}(t,x,y)
 \leq 
\frac{C}{\sqrt[2q(t)]{m({B_{x}(r_1)})m({B_{y}(r_1)})}}
\\
\cdot
\frac{(1\vee S^{-2}(\sqrt{\rho(B_x(r_1),B_y(r_1))^2S^2+t^2}-t))^{\frac{n}{2}}}{\sqrt[\frac{2q(t)}{q(t)-1}]{m(B_x(\sqrt t))m(B_y(\sqrt t))}}
 \euler^{-tS^{-2}\zeta(\rho(B_x(r_1),B_y(r_1))S/t)}.
\end{multline*}
\end{corollary}
}

\eat{
\begin{corollary}
Let  
$n>2$,  $r_0\geq 128S$. Assume $S(n,r_0,\infty)$ in every vertex of $X$. Then there is a constant $C>0$ such that, for all $r_1\geq r_0$ and $t\geq 2r_1^2$, $x\in X$, we have 
\[
\mathcal{P}_{r_1}(t,x,x)
 \leq 
\frac{C}{\sqrt[q(t)]{m({B_{x}(r_1)})})}
\frac{1}{\sqrt[\frac{q(t)}{q(t)-1}]{m(B_x(\sqrt t))}}.
\]
\end{corollary}
}

\section{Averaged heat kernel bounds imply Faber-Krahn}\label{sec:AVtoFK}

In the present section we derive relative Faber-Krahn inequalities for sets of sufficient measure from averaged Gaussian upper heat kernel bounds. Section~\ref{sec:ev_av} provides a new general lower bound for Dirichlet eigenvalues of subsets in terms of averaged semigroup norms and the averaged heat kernel. In Section~\ref{sec:AVtoFK} we derive the corresponding main result.

\subsection{Eigenvalue bounds in terms of averaged  semigroup norms}\label{sec:ev_av}

Our first observation is a lower bound on the first Dirichlet eigenvalue of a subset of a graph in terms of averaged heat semigroup norms.

\begin{theorem}\label{lem:ev_bound_abstract1}For all $U\subset X$,  $r\geq 0$, and $q\colon [0,\infty)\to [1,\infty]$, we have 
\[
\lambda(U)
\geq \sup_{t\ge r^2}\ 
\frac{1}{2t}
\ln\left(m(U)^{1-\frac{1}{q(t)}}\left(\fint_{t-r^2}^{t+r^2}\|UP_{\tau/2}U\|_{\frac{2q(t)}{2q(t)-1},2q(t)}^{q(t)}\drm \tau \right)^{\frac{1}{q(t)}}\right)^{-1},
\]
\end{theorem}
\begin{remark}If $q=\infty$, then the above estimate reduces to 
\[
\lambda (U)\geq \sup_{t\geq r^2}\ \frac{1}{2t}\ln\left(m(U)\sup_{\tau\in[(t-r^2)/2,(t+r^2)/2]}\|UP_\tau U\|_{1,\infty}\right)^{-1}.
\]
To see this, note that $(\fint_I |f|^p)^{1/p}\to \|f\restriction_I\|_\infty$ if $p\to \infty$ for a compact interval~$I$, $\|A\|_{p',p}\to \|A\|_{1,\infty}$ if $p\to\infty $ since $p'=p/(p-1)\to 1$ as $p\to\infty$, and that the heat semigroup is substochastic.
\end{remark}
\begin{proof}
Let $\phi$ denote the positive eigenfunction associated to $\lambda(U)$ normalized such that $\|\phi\|_{\frac{q}{q-1}}=1$. Since $q\geq 1$, H\"older's inequality implies $1\leq m(U)^{1-\frac{1}{2q}}\|\phi\|_{2q}$ and hence
\[
\|UP_{t/2}U\|_{\frac{2q}{2q-1},2q}=\sup_{\|\varphi\|_{\frac{2q}{2q-1}}=1}\|P_t\varphi\|_{\ell^{2q}(U)}\geq \|P^U_t\phi\|_{2q}=\euler^{-\lambda(U)t}\|\phi\|_{2q}\geq \euler^{-\lambda(U)t}m(U)^{\frac{1}{2q}-1}.
\]
Hence, for all $U\subset X$ and $t\geq r^2\geq 0$, we have 
\begin{multline*}
 \euler^{-2t\lambda(U)}\leq \euler^{-(t+r^2)\lambda(U)}\leq \left(\fint_{t-r^2}^{t+r^2}\euler^{-\tau\lambda(U)q(t)}\drm \tau\right)^{\frac{1}{q(t)}}
\\
\leq m(U)^{1-\frac{1}{2q(t)}}\left(\fint_{t-r^2}^{t+r^2}\|UP_{\tau/2}U\|_{\frac{2q(t)}{2q(t)-1},2q(t)}^{q(t)}\drm \tau \right)^{\frac{1}{q(t)}}
\end{multline*}
Rearranging yields the claim.
\end{proof}

In the case that the subsets we want to deal with are contained in a geodesic ball, we can compare the Dirichlet eigenvalue to the averaged heat kernel in the ball. 

\begin{theorem}\label{thm:ev_bound_abstract2}Let $o\in X$, $r_2\geq r_1\geq 0$, and a function $q\colon [0,\infty)\to [1,\infty]$ be given. Then, 
for all $U\subset B(r_2)\subset X$, we have 
\[
\lambda(U)
\geq \sup_{t\ge r_1^2}\ 
\frac{1}{2t}
\ln\left(m(U)^{1-\frac{1}{2q(t)}}\left(\frac{r_2}{r_1}\right)^{\frac{2}{q(t)}}\mathcal{P}_{r_2}(t,o,o)\right)^{-1}.
\]
\end{theorem}
\begin{proof}
For $r_2\geq r_1$, we have 
\[
\left(\fint_{t-r_1^2}^{t+r_1^2}\|UP_{\tau/2}U\|_{\frac{2q(t)}{2q(t)-1},2q(t)}^{q(t)}\drm \tau \right)^{\frac{1}{q(t)}}
\leq 
\left(\frac{r_2^2}{r_1^2}\right)^{\frac{1}{q(t)}}\left(\fint_{t-r_2^2}^{t+r_2^2}\|UP_{\tau/2}U\|_{\frac{2q(t)}{2q(t)-1},2q(t)}^{q(t)}\drm \tau \right)^{\frac{1}{q(t)}}
\]
Since $P_t$ is positivity preserving, for all $U\subset B(r_2)$, we have 
\[
\|UP_tU\|_{2q',2q}\leq \|B(r_2)P_tB(r_2)\|_{2q',2q}.
\]
Applying this estimate to the right-hand side above, the claim follows from Lemma~\ref{lem:ev_bound_abstract1}, and the definition of $\mathcal{P}_{r_2}(t,o,o)$ .
\end{proof}

\subsection{From averaged semigroup norms to Faber-Krahn}\label{sec:av_fk}

In this section we obtain relative Faber-Krahn inequalities for subsets in balls whose measure is a given portion of the measure of the ball and a variable dimension. To this end, we modify the classical approach from \cite{Grigoryan-09} and combine the general bound for Dirichlet eigenvalues in Theorem~\ref{thm:ev_bound_abstract2} with the Gaussian bound on the averaged heat kernel. The restriction of the subsets and the variable dimension result from the fact that we only assume the averaged heat kernel bounds to hold on large scales. 
\\

In order to conclude the result as in the classical case, we need the reverse volume doubling property.
Recall the following result which is the adaption to graphs with unbounded geometry from the manifold setting \cite{Grigoryan-94}. 
\begin{proposition}[reverse doubling,{\cite[Lemma~7.4]{Rose24}}]\label{prop:reversedoubling}
$32S\leq 8\hat r\leq r\leq \diam X/2$, $o\in X$, constants $n>0$, $C_D\geq 1$, and assume 
$V(\hat r,r,C_D,n) $ in $B_o(r)$.
Then, we have 
\[
m(B_o(r_2))\geq \frac12\left(\frac{r_2}{r_1}\right)^{\eta} m(B_o(r_1)),\quad 4\hat r\leq r_1\leq r_2\leq r/2,
\]
where
$
\eta=(2C_D)^{-21n}.
$
\end{proposition}

The next theorem is our second main result. It gives relative Faber-Krahn inequalities for sets with prescribed measure and a variable dimension.

\begin{theorem}\label{thm:main_hk_FK}Let $o\in X$, $n>0$, $r\geq 0$, $C\geq 1$, $\varepsilon\in(0,1)$, $q\colon [r,\infty)\to (0,\infty)$ monotone increasing with $q(t)\to\infty$ as $t\to\infty$, and
\[
A\geq  
\left((2C)^{3(2C)^{21n}}\cdot 1\vee  m(B(r)^{\frac{(2C)^{21n}}{q(r)}}\right)^\frac{q(r^2)}{q(r^2)-1}.
\] 
Assume \[
\mathcal{P}_{r}(t,o,o)
 \leq 
\frac{C}{\sqrt[\frac{q(t)}{q(t)-1}]{m(B(\sqrt t))}}
\]
for all $r^2\leq t\leq 2A^4r^2$ and
\[
\frac{m(B_x(s_2))}{m(B_x(s_1))}\leq C\left(\frac{s_2}{s_1}\right)^n
\]
for all $x\in B(2Ar)$ and $r\leq s_1\leq s_2\leq 2A^2r$.
Then, for all $U\subset B(r)$ satisfying $m(U)\geq \varepsilon m(B(r))$, we have 
\[
\lambda(U)
\geq 
\frac{C_n}{A^4r^2}
\left(
\frac{m(B(r))}{m(U)}\right)^{\frac{2}{N}},
\]
where $C_n=\frac{1}{C^{\frac{4}{n}}2^{5+\frac{4}{n}}}$ and 
\[
N\geq n\frac{q(r^2)}{q(r^2)-1}+ \frac{\ln\frac{1}{\epsilon}}{\ln A}+\frac{1}{q(r^2)-1}\ln(1\vee m(B(r))).
\]
\end{theorem}

\begin{proof}Note that $A\geq 1$. Set 
\[\delta:=\frac{1
}{\euler C^21\vee ( m(B(r))^{1/q(r^2)}}\in(0,1).\] 
Fix $U\subset B(r)$ satisfying $m(U)\geq \epsilon m(B(r))$. 
\\

\noindent
\underline{$1^{\mathrm{st}}$ case:}
\[
\left(\frac{m(U)}{m(B(r))}\right)^{\frac{q(r^2)-1}{q(r^2)}}\leq \delta.
\] 
For all $t\in[r^2,A^2r^2]$, volume doubling yields
\begin{multline*}
\mathcal{P}_{r}(t,o,o)
\leq \frac{C}{m(B(\sqrt t))^\frac{q(t)-1}{q(t)}}
\\\leq C{C}^\frac{q(t)-1}{q(t)}\left(\frac{1}{m(B(Ar))}\left(\frac{Ar}{\sqrt t}\right)^{n}\right)^{\frac{q(t)-1}{q(t)}}
\leq C^2\left(\frac{1}{m(B(r))}\left(\frac{Ar}{\sqrt t}\right)^{n}\right)^{\frac{q(t)-1}{q(t)}}.
\end{multline*}
For all $t\in[r^2,A^2r^2]$, we have \[n\frac{q(r^2)-1}{q(r^2)}\le n\frac{q(t)-1}{q(t)} \le n \frac{q(A^2r^2)-1}{q(A^2r^2)}\leq n \leq n'=\frac{q(r^2)-1}{q(r^2)}N.\]
Hence, for all $t\in[r^2,A^2r^2]$, we get 
\begin{multline*}
m(U)^{\frac{q(t)-1}{q(t)}+\frac{1}{2q(t)}}P_{r}(t,o,o)
\leq 
C^21\vee(m(B(r)))^{\frac{1}{q(r^2)}}\left(\frac{m(U)}{m(B(r))}\right)^\frac{q(t)-1}{q(t)}\left(\frac{Ar}{\sqrt t}\right)^{n\frac{q(t)-1}{q(t)}}
\\
\leq 
\frac{1}{\euler\delta}\left(\frac{m(U)}{m(B(r))}\right)^{\frac{q(r^2)-1}{q(r^2)}}\left(\frac{Ar}{\sqrt t}\right)^{n'}
\end{multline*}
where we used $m(U)/m(B(r))\le 1$ and $r\geq \sqrt t$ in the last estimate.
Plugging these two estimates into the lower bound of $\lambda(U)$ obtained in Lemma~\ref{lem:ev_bound_abstract1} implies
\begin{multline*}
\lambda(U)
\geq \sup_{A^2r^2\geq t\ge r^2}\ 
\frac{1}{2t}
\ln\left(m(U)^{\frac{q(t)-1}{q(t)}+\frac{1}{2q(t)}}\mathcal{P}_{r}(t,o,o)\right)^{-1}
\\
\geq 
\eat{\sup_{A^2r^2\geq  t\ge r^2}
\frac{1}{2t}
\ln
\left(
\frac{CC_D }{1\wedge(\epsilon m(B(r)))^{\frac{1}{q(r)}}}
\left(
\frac{m(U)}{m(B(r))}
\right)^{\frac{q(r)-1}{q(r)}}
\left(
\frac{Ar}{\sqrt t}
\right)^{n'}
\right)^{-1}
\\=}
\sup_{A^2r^2\geq  t\ge r^2}
\frac{1}{2t}
\ln\left(
\euler\delta
\left(\frac{m(B(r))}{m(U)}\right)^{\frac{q(r^2)-1}{q(r^2)}}\left(\frac{\sqrt t}{Ar}\right)^{n'}\right)
\end{multline*}
Set
\[
t_0=\left(\frac{1}{\delta}\left(\frac{m(U)}{m(B(r))}\right)^{\frac{q(r^2)-1}{q(r^2)}}\right)^{\frac{2}{n'}}A^2r^2
\]
The time $t=t_0$ satisfies $t_0\leq A^2r^2$. Moreover, $t_0\geq r^2$. Indeed, since $\delta\leq 1$, $\tfrac{m(U)}{m(B(r))}\geq \varepsilon$, 
 and 
$n'\geq \frac{\ln\frac{1}{\varepsilon}}{\ln A}$,
we have
\[
t_0\geq \varepsilon^{\frac{q(r^2)-1}{q(r^2)}\frac{2}{n'}}A^2 r^2
\geq r^2.
\]
Hence, we can plug $t_0$ into the supremum in the lower bound for $\lambda(U)$. Recall the definition of $\delta=1/(\euler C^21\vee ( m(B(r))^{1/q(r^2)})$
to obtain
\begin{multline*}
\lambda(U)\geq 
\frac{\delta^{\frac{2}{n'}}}{2A^2r^2}\left(\frac{m(B(r))}{m(U)}\right)^{\frac{2}{n'}\frac{q(r^2)-1}{q(r^2)}}
= \frac{(\euler C^2)^{-\frac{2}{n'}}}{2A^2r^2}\left(\frac{m(B(r))}{m(U)}\right)^{\frac{2}{n'}\frac{q(r^2)-1}{q(r^2)}}
\frac{1}{1\vee ( m(B(r))^{2/n'q(r^2)}}
\end{multline*}
Since $n'\geq n$ we have for the first factor 
\[
\frac{(\euler C^2)^{-\frac{2}{n'}}}{2A^2r^2}\geq \frac{(2^2 C^2)^{-\frac{2}{n}}}{2A^2r^2}=\frac{C^{-\frac{4}{n}}2^{-1-\frac{4}{n}}}{A^2r^2}.
\]
Regarding the last factor, since
$n'\geq \frac{1}{q(r^2)}\ln(1\vee m(B(r)))$, we have
\[
\frac{1}{1\vee ( m(B(r))^{2/n'q(r^2)}}
\geq \frac{1}{1\vee ( m(B(r))^{2/\ln(1\vee m(B(r)))}}
=\frac{1}{\euler^2}\geq \frac1{2^4}
\]
Hence, by setting $a(n):=C^{-\frac{2}{n}}2^{-5-\frac{4}{n}}$, we get 
\[
\lambda(U)\geq 
\frac{a(n)}{A^2r^2}\left(\frac{m(B(r))}{m(U)}\right)^{\frac{2}{n'}\frac{q(r^2)-1}{q(r^2)}}=\frac{a(n)}{A^2r^2}\left(\frac{m(B(r))}{m(U)}\right)^{\frac{2}{N}}.
\]

\noindent
\underline{$2^\mathrm{nd}$ case:}
\[
\left(\frac{m(U)}{m(B(r))}\right)^{\frac{q(r^2)-1}{q(r^2)}}>\delta.
\]
Note \[A\geq\left(2\euler C^2 1\vee m(B(r))^{1/q(r^2)}\right)^{(2C)^{21n} \frac{q(r^2)} {q(r^2)-1}}\geq\left(\frac{2}{\delta^{\frac{q(r^2)} {q(r^2)-1}}}\right)^{1/\eta}\] where $\eta=(2C)^{-21n}$. 
By Proposition~\ref{prop:reversedoubling}, we have the reverse volume doubling property for radii up to $A^2r\geq Ar\geq r$.
Since $m(U)\leq m(B(r))$ and $\delta\in(0,1)$, we obtain
\begin{multline*}
\left(\frac{m(U)}{m(B(Ar))}\right)^{\frac{q(A^2r^2)-1}{q(A^2r^2)}}
\leq 
\left(\frac{m(U)}{\frac12A^\eta m(B(r))}\right)^{\frac{q(A^2r^2)-1}{q(A^2r^2)}}
\\\leq
\left(\delta^{\frac{q(r^2)} {q(r^2)-1}} \frac{m(U)}{m(B(r))}\right)^{\frac{q(A^2r^2)-1}{q(A^2r^2)}}
\leq\delta^{\frac{q(r^2)} {q(r^2)-1}\frac{q(A^2r^2)-1}{q(A^2r^2)}}\leq\delta
\end{multline*}
Hence, the second case can be handled by following the first case line by line and replacing the involved radii appropriately.
The volume doubling condition and the estimate on $\mathcal{P}$ hold for radii and times up to $A^2r$ and $A^4r^2$, respectively. We carry out the estimates involving $\mathcal{P}$ where $r$ is replaced by $Ar$ to get a lower bound on $\lambda(U)$. We use the bound
\[
n'\geq {\frac{1}{\ln{A}}}\ln\left(\frac{1}{\varepsilon}\right)\geq {\frac{1}{\ln{A}}}\frac{q(A^2r^2)-1}{q(A^2r^2)}\ln\left(\frac{1}{\varepsilon}\right)\geq {\frac{1}{\ln{A}}}\ln\left(\frac{\delta}{\varepsilon^{\frac{q(A^2r^2)-1}{q(A^2r^2)}}}\right)
\]  in order to control the corresponding value for $t_0$. These adaptions imply 
\[
\lambda(U)\geq 
\frac{\delta^{\frac{2}{n'}}}{2A^4r^2}\left(\frac{m(B(Ar))}{m(U)}\right)^{\frac{2}{n'}\frac{q(A^2r^2)-1}{q(A^2r^2)}}\geq 
\frac{a(n)}{A^4r^2}\left(\frac{m(B(r))}{m(U)}\right)^{\frac{2}{N}}
\]
where we used $A\geq1$ in the last estimate.
This yields the claim.
\end{proof}

Theorem~\ref{thm:hk_FK_global} is the  global version of Theorem~\ref{thm:main_hk_FK}. It is formulated as follows.

\begin{theorem}
Let $r_0,C>0$. Assume that we have 
\[
\mathcal{P}_{r}(t,x,y)
\\
 \leq 
C
\frac{(1\vee S^{-2}(\sqrt{\rho(B_x(r),B_y(r))^2S^2+t^2}-t))^{\frac{n}{2}}}{\sqrt[\frac{2q(t)}{q(t)-1}]{m(B_x(\sqrt t))m(B_y(\sqrt t))}}
 \euler^{-tS^{-2}\zeta(\rho(B_x(r),B_y(r))S/t)}
\]
for all $t\geq r^2\geq r_0^2$ and $x,y\in X$, 
and the volume comparison estimate
\[
\frac{m(B_x(R))}{R^{n}}
\leq C 
\frac{m(B_x(r))}{r^{n}}, \quad r_1\leq r\leq R,
\]
for all $x\in X$.
Then, for all $\varepsilon\in(0,1)$, $r\geq 1\vee r_0$, and $U\subset B(r)$ satisfying $m(U)\geq \varepsilon m(B(r))$, we have 
\[
\lambda(U)
\geq 
\frac{C_n'}{r^{2+\varepsilon}}
\left(
\frac{m(B(r))}{m(U)}\right)^{\frac{2}{N}}\left[1\wedge \frac{1}{m(B(r))}\right] ^{\frac{1}{C_n'q(r^2)}}
\]
where $C_n'=C_n'(C,n)>0$ and 
\[
N\geq n\frac{q(r^2)}{q(r^2)-1}+ \frac{\ln \frac{1}{\varepsilon}}{\varepsilon}\frac{1}{\ln r}+\frac{1}{q(r^2)-1}\ln(1\vee m(B(r)))
\]
\end{theorem}
\begin{proof}
Choose
\[
A= r^{\frac{\varepsilon}4}\vee {(2C)^{3(2C)^{21n}}}\cdot {1\vee ( m(B(r))^{\frac{(2C)^{21n}}{q(r^2)}}}
\] 
to obtain for $r\geq 1$
\begin{multline*}
\lambda(U)
\geq 
\left[\frac{1}{r^\varepsilon}\wedge \frac{1}{(2C)^{12(2C)^{21n}} 1\vee m(B(r))^{\frac{4(2C)^{21n}}{q(r^2)}}}
\right]
\frac{C_n}{r^2}
\left(
\frac{m(B(r))}{m(U)}\right)^{\frac{2}{N}}
\\
\geq 
\left[\frac{1}{(2C)^{12(2C)^{21n}} 1\vee m(B(r))^{\frac{4(2C)^{21n}}{q(r^2)}}}\right]
\frac{C_n'}{r^{2+\varepsilon}}
\left(
\frac{m(B(r))}{m(U)}\right)^{\frac{2}{N}},
\end{multline*}
This yields the claim. 
\end{proof}

\begin{remark}Alternatively to our proof strategy of Theorem~\ref{thm:main_hk_FK}, one might be tempted to mimick the pointwise approach by estimating the averaged semigroup in a subset with a covering of  balls of least possible radius and use a Vitali covering argument. This leads indeed to a much worse estimate involving the ratio of diameter and least radius which cannot be mitigated within the present general strategy.
\end{remark}

{\small

\bibliographystyle{alpha}
\bibliography{Davies_ref.bib}
}

\end{document}